\documentclass[final,leqno]{siamltex}

\usepackage{amssymb}
\usepackage{booktabs}
\usepackage{lmodern}
\usepackage{mathtools}
\usepackage{nth}
\usepackage{graphicx}
\usepackage{subfigure}
\usepackage{hyperref}

\def\A{{\bf A}}
\def\a{{\bf a}}
\def\B{{\bf B}}
\def\b{{\bf b}}

\def\D{{\bf D}}

\def\H{{\bf H}}
\def\G{{\bf G}}
\def\I{{\bf I}}

\def\L{{\bf L}}
\def\R{{\bf R}}
\def\X{{\bf X}}
\def\Y{{\bf Y}}
\def\Q{{\bf Q}}
\def\BP{{\bf P}}
\def\p{{\bf p}}

\def\S{{\bf S}}
\def\T{{\mathcal T}}
\def\x{{\bf x}}
\def\y{{\bf y}}

\def\Z{{\bf Z}}
\def\M{{\bf M}}

\def\N{{\bf N}}
\def\n{{\bf n}}
\def\U{{\bf U}}

\def\V{{\bf V}}
\def\v{{\bf v}}
\def\W{{\bf W}}

\def\0{{\bf 0}}
\def\1{{\bf 1}}
\def\wrt{{w.r.t.\ }}

\def\prodod{\otimes^\downarrow}

\newcommand{\SLG}{\mathfrak{SL}}
\newcommand{\GLG}{\mathfrak{GL}}
\newcommand{\OG}{\mathfrak{O}}
\newcommand{\HG}{\mathfrak{H}}
\newcommand{\SOG}{\mathfrak{SO}}
\newcommand{\UG}{\mathfrak{U}}
\newcommand{\UUTG}{\mathfrak{UUT}}
\newcommand{\LUTG}{\mathfrak{LUT}}

\newcommand{\Poly}{\operatorname{Poly}}

\newcommand{\Rasterize}{\operatorname{Rasterize}}

\newcommand{\eqdef}{\overset{\mathrm{def}}{=\joinrel=}}

\def\nnz#1{\|#1\|_0}

\def\norm#1{\|#1\|}

\def\Si{\mbox{\boldmath$\Sigma$\unboldmath}}

\def\Lam{\mbox{\boldmath$\Lambda$\unboldmath}}

\def\AM{{\mathcal A}}
\def\BM{{\mathcal B}}

\def\TM{{\mathcal T}}

\def\XM{{\mathcal X}}
\def\YM{{\mathcal Y}}
\def\MM{{\mathcal M}}

\def\GM{{\mathfrak G}}
\def\gm{{\mathfrak g}}
\def\RB{{\mathbb R}}
\def\FB{{\mathbb F}}

\def\ZM{{\mathcal Z}}

\def\fold{\operatorname{fold}}
\def\index{\operatorname{index}}
\def\conj#1{#1^{\mathrm{c}}}

\def\tr{\operatorname{tr}}
\def\rk{\operatorname{rank}}
\def\sorank{\operatorname{rank}_{so}}
\def\diag{\operatorname{diag}}

\def\dg{\operatorname{dg}}
\def\vect{\operatorname{vec}}

\def\arginf{\mathop{\rm arginf}}

\newtheorem{example}[theorem]{Example}
\newtheorem{remark}[theorem]{Remark}

\newtheorem{algorithm}[theorem]{Algorithm}

\title{Group Orbit Optimization: A Unified Approach to Data Normalization}

\author{Shuchang Zhou \footnotemark[3]
\thanks{Megvii Inc., Beijing, China, 100083.}
\and Zhihua Zhang\thanks{Department of Computer Science and Engineering, Shanghai Jiao Tong University, Shanghai, China 200240.} \and Xiaobing Feng\thanks{State Key Laboratory of Computer
Architecture, Institute of Computing Technology,
Chinese Academy of Sciences, Beijing, China, 100190.}}

\begin{document}
\maketitle


\begin{abstract}
In this paper we propose and study
an optimization problem over a matrix group orbit that we call \emph{Group Orbit Optimization} (GOO).
We prove that GOO can be used to induce matrix decomposition techniques such as singular value decomposition (SVD), LU decomposition,
QR decomposition, Schur decomposition and Cholesky decomposition, etc.
This gives rise to a unified framework for matrix decomposition and  allows us to bridge  these
matrix decomposition methods. Moreover, we generalize GOO for tensor
decomposition. As a concrete application of GOO, we devise a new data decomposition method over a special linear group to normalize point cloud data. Experiment results show that our normalization method is able to obtain recovery
well from distortions like shearing, rotation and squeezing.
\end{abstract}



\pagestyle{myheadings}
\thispagestyle{plain}
\markboth{SHUCHANG ZHOU, ZHIHUA ZHANG AND XIAOBING FENG}{DATA NORMALIZATION
BY GROUP ORBIT OPTIMIZATION}

\section{Introduction}
\label{sec:intro}

Real world data often contain some degrees of freedom that might be redundant. Matrix
decomposition~\cite{Golub:1996,DemmelBook:1997,TrefethenBook:1997} is an
important tool in machine learning and data mining to normalize data.
A prominent example of data normalization by matrix decomposition is principal
component analysis (PCA). When the given point cloud is represented as a matrix
with each row being coordinates of points, PCA removes the degree of freedom in
translation and rotation of the point cloud with the help of singular value decomposition (SVD) on the matrix.
The selection of particular matrix decomposition corresponds to which
degrees of freedom we would like to remove. In the PCA example, SVD extracts an
orthonormal basis that makes the normalized data invariant to rotation.

There are cases when other degrees of freedom exist in data. For example, planar
objects like digits, characters or iconic symbols, often look distorted in
photos because the camera sensor plane may not be parallel to the plane carrying
the objects.
Therefore in this case, the degrees of freedom we would like to eliminate from
data are homography transforms \cite{hartley2003multiple}, which can be
approximated as combination of translation, rotation, shearing and squeezing when the planar objects are
sufficient far away relative to their size.
However, PCA is not applicable to eliminate these degrees of freedom, because
the normalized form found with PCA is not invariant under shearing and
squeezing. In general, based on the property of data, we would need new data
normalization methods that can uncover invariant structures depending on the
degrees of freedom we would like to remove.


\begin{figure}
\begin{center}
\includegraphics[height=32mm, width=100mm]{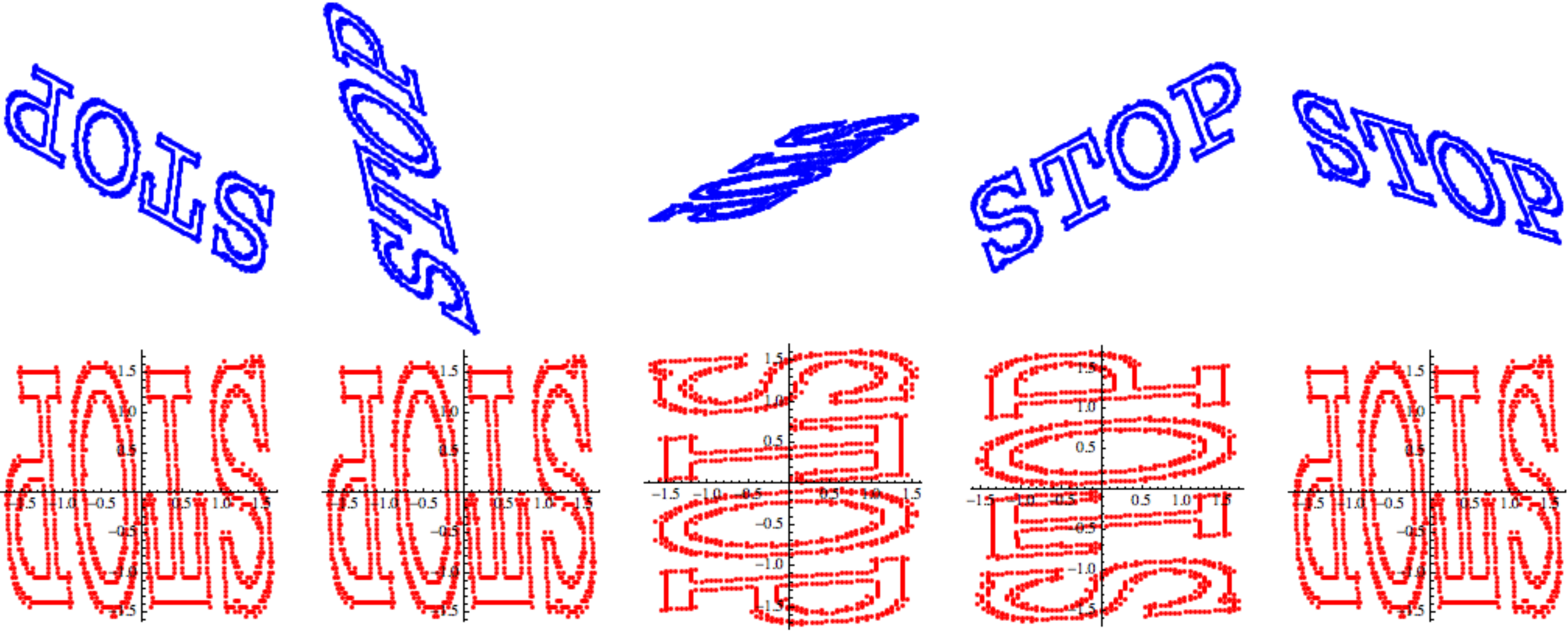}
\end{center}
   \caption{Normalization by optimization over orbit generated by special
   linear group $\SLG(2)$ for 2D point clouds. The first row contains point
   clouds before normalization; the second row consists of corresponding point clouds after
   normalization for each entry in the first row.
   It can be observed that point clouds in the second row are approximately the
   same, modulo four orientations (rotated clockwise by angle of 0,
   $\frac{\pi}2$, $\pi$, $\frac{3\pi}2$).}
\label{fig:special_linear}
\end{figure}

\begin{figure}
\begin{center}
\includegraphics[height=32mm, width=100mm]{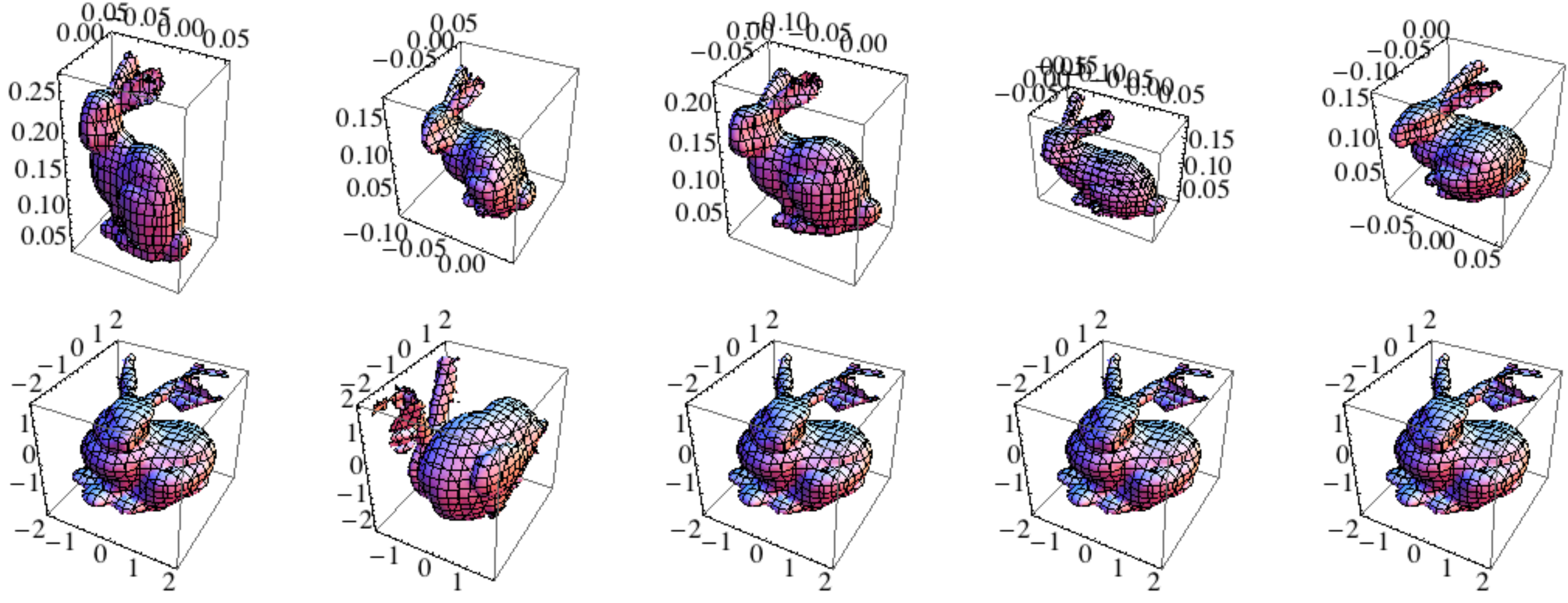}
\end{center}
   \caption{Normalization by optimization over orbit generated by special
   linear group $\SLG(3)$ for 3D point clouds. The first row contains point
   clouds before normalization.
   In particular, the ``rabbits'' are of different shapes and sizes.
   The second row consists of corresponding point clouds after normalization
   for each entry in the first row.
   It can be observed that point clouds in the second row are approximately
   the same, modulo different orientations of the same shape.}
\label{fig:special_linear_3d}
\end{figure}

In this paper we study the cases when degrees of freedom to be removed have a
group structure $\GM$ when combined. Under such a condition, a data matrix $\X$
can be mapped to its quotient set $\X/\sim$ by the equivalence relation $\sim$
defined as \[x_1\sim x_2\iff \exists g\in\GM, x_1 = g x_2 \text{.}\] We call the
elements of quotient set $\hat{\X} \in \X/\sim$ canonical forms of data, as they
are invariant with respect to (w.r.t.) group actions $g\in\GM$. An important
example of using the quotient set is the shape space method
\cite{dryden1998statistical}, which works in the quotient space of rotation
matrix and is closely related to PCA and SVD.

Here and later, we restrict ourselves to the case when $\GM$
is a matrix group and when the group acts by simple matrix product. The
quotient mapping $\X\rightarrow \hat{\X}$ can then be represented in the form of
matrix decomposition:
\[\X = \G\hat{\X}\text{, } \G\in\GM\text{.}\]

Instead of constructing separate algorithms for different $\GM$, we use an
optimization process to induce corresponding matrix decomposition techniques.
In particular, given a data matrix $\M$, we consider a group orbit optimization
(GOO) problem as follows:
\begin{align}
\label{eq:goo}
\inf_{\G \in \GM} \; \phi(\G \M),
\end{align}
where $\phi:
\FB^{n_1\times n_2}\rightarrow \RB$ is a cost function and $\FB$ is some number field.

In Section~\ref{sec:prelim} we present 
several special classes of cost functions, which are used to construct new
formulations for several matrix decompositions including SVD, Schur, LU,
Cholesky and QR in Section~\ref{sec:main}. As an application, in Section~\ref{sec:normal} we illustrate how to use GOO to normalize low dimensional point cloud data over a special linear group.
Experiment results for two-dimensional and three-dimensional point cloud are
given in Figure~\ref{fig:special_linear} and Figure~\ref{fig:special_linear_3d}.
It can be observed that the effect of rotation, shearing and squeezing in
data has been mostly eliminated in the normalized point clouds. The detail of
this normalization is explained in Section~\ref{sec:normal}.

The  GOO formulation also allows us to construct
generalizations of some matrix decompositions to tensor. Real world
data have tensor structure when some value depends on multiple factors. For
example, in an electronic-commerce site, user preferences in different brands
form a matrix. As such preferences change over time, the time-dependent
preferences form a \nth{3} order tensor. As in the matrix case, tensor
decomposition techniques~\cite{Kolda:2009:TDA:1655228.1655230,
krishnamurthy2013low} aim to eliminate degrees of freedom in data while respecting the tensor structure of data. In
Section~\ref{sec:tensor}, we use GOO to induce tensor decompositions that can be
used for normalizing tensor. In the unified framework of GOO, the GOO inducing
tensor decomposition when applied to a \nth{2} order tensor, is exactly the same
as the GOO inducing matrix decomposition, when the same group and cost function
is used for both GOO problems.

The remainder of paper is organized as follows. Section~\ref{sec:notation} gives
notation used in this paper. Section~\ref{sec:prelim} defines several
properties for describing the cost function used in defining GOO to induce
matrix and tensor decompositions. Section~\ref{sec:main} studies GOO formulations
that can induce SVD, Schur, LU, Cholesky, QR, etc. Section~\ref{sec:tensor}
demonstrates how to use GOO to induce tensor decompositions and prove a few inequalities relating
a few forms of GOO. Section~\ref{sec:normal} demonstrates how to normalize
point cloud data distorted by rotation, shearing and squeezing with GOO over
the special linear group.
Section~\ref{sec:exps} presents numerical algorithms and examples of matrix
decomposition, point cloud normalization and tensor decomposition.
Finally, we conclude the work in Section~\ref{sec:conclusion}.

\section{Notation}
\label{sec:notation}

\subsection{Matrix operation notation}
In this paper, 
we let $\I_r$ denote the $r\times r$ identity matrix. Given an $n{\times}m$
matrix $\X=[x_{ij}]$, we denote $|\X|=[|x_{ij}|]$ and $\vect(\X)=[x_{11},
\ldots, x_{n1}, x_{12}, \ldots, x_{nm}]^\top$. The $\ell_p$-norm of $\X$ is
defined by \[\|\X\|_p\eqdef (\sum_{ij} |x_{ij}|^p)^{\frac1p}\] for $p\ge 0$. Note that
we abuse the notation a little bit as $\|\X\|_p$ is not a norm when $p<1$.
When $p=2$, it is also called the Frobenius norm and usually denoted by
$\|\X\|_F$.
When applied to vector $\x$, $\|\x\|_2$ is the  $\ell_2$-norm  and it is
shortened as $\|\x\|$. The dual norm of the $p$-norm where $p\ge 1$ is
equivalent to the $q$-norm, where $\frac{1}{p}+\frac{1}{q}=1$.
We let $\|\X\|_{*p}$ denote the Schatten $p$-norm; that is,  it is the $\ell_p$ norm of the vector of
the singular values of $\X$.

Assume that $\mathbb{F}$ is some number field.
Let $\conj{\X}$ be the complex conjugate
of $\X$, and $\X^*$ be the complex conjugate transpose of $\X$.
Let $\dg(\M)$ be  a vector consisting of the diagonal entries of $\M$, and
$\diag(\v)$ be a matrix with $\v$ as its diagonals.

Given two matrices $\A$ and $\B$, $\A\odot \B$ is
their Hadamard product and
$\A \otimes \B$ is the Kronecker product. Similarly, $\x\otimes \y$ is the Kronecker product of vectors $\x$ and
$\y$. For groups $\GM_1$ and $\GM_2$, we denote group
$\{\G_1\otimes\G_2:\, \G_1\in\GM_1, \G_2\in\GM_2\}$ as $\GM_1\otimes\GM_2$. The
Kronecker sum for two square matrices $\A\in\FB^{m\times m}, \B\in\FB^{n\times
n}$ is defined as \[\A\oplus\B = \A\otimes \I_n + \I_m \otimes \B \text{.}\]


\begin{definition}
A matrix $\A \in \mathbb{F}^{m\times n}$ is said to be pseudo-diagonal if
there exist  permutation matrices $\BP$ and $\Q$ such that $\BP \A \Q^\top$
is diagonal.
\end{definition}


\begin{remark}
Note that a diagonal matrix is also pseudo-diagonal.
\end{remark}

\begin{lemma}
Given a pseudo-diagonal matrix $\A$, we have that
\begin{enumerate}
\item[\emph{(i)}] $\A^*\A$, $\A \A^*$, $\A^\top\A$ and $\A \A^\top$
are diagonal.
\item[\emph{(ii)}] There exists a row permutation matrix $\BP$ such that $\BP \A$ is
diagonal.
\item[\emph{(iii)}] There exists a row permutation matrix $\BP$ such that $\A \BP^\top$
is diagonal.
\end{enumerate}
\end{lemma}

We let $\Poly{(\M)}$ be the polyhedral formed by points with coordinates being
rows of $\M$, and $\mu(\Poly{(\M)})$ be the Lebesgue measure of $\Poly{(\M)}$.
We let $\Rasterize(\Poly{(\M)})$ be a matrix $\Z$ where $z_{ij}$ is the
image pixel value at coordinate $(i, j)$ of image rasterized from polyhedral
$\Poly{(\M)}$ with unit grid.

\subsection{Tensor operation notation}

The notation of tensor operations used in this paper mostly follows that of
\cite{Kolda:2009:TDA:1655228.1655230}.
Given an order-$k$ tensor $\XM \in \mathbb{F}^{n_1\times n_2\times \ldots \times n_k}$
and $k$ matrices $\{\U_i\}_{i=1}^k$ where $ \U_i \in
\mathbb{F}^{m_i\times n_i}$, we define $\times_k$ to be the inner product over the
$k$-th mode. That is, if $\YM = \XM \times_a \U_a\in \mathbb{F}^{n_1\times
n_2\times \ldots \times n_{a-1} \times m_a \times n_{a+1} \times \ldots \times
n_k}$, then \[y_{i_1\cdots i_{a-1}j i_{a+1}\cdots i_k} = \sum_{i_a=1}^{n_a} x_{i_1i_2\cdots i_k} u_{j
i_a}.
\]
For shorthand,
we denote \[\prod_i \XM \U_i \eqdef \XM \times_1 \U_1
\times_2 \U_2 \cdots \times_k \U_k \textrm{.}\]
Here $\YM = \prod_i \XM \U_i$ when $\forall i, m_i = n_i$ is also known as
the Tucker decomposition in the literature \cite{tucker1966some}. With this notation,
the SVD of a real matrix $\M=\U_1\Si \U_2^\top$ can
be written as \[\M= \Si \times_1 \U_1\times_2 \U_2=\prod_{i=1}^2
\Si \U_i\text{.}\]
Using the vectorization operation for tensor, we have
\[
\vect{(\prod_i \XM \U_i)} = [\U_n \otimes\U_{n-1}\otimes\dotsb\otimes\U_{1}] \vect{(\XM)}
\eqdef \prodod_i\U_i \vect{(\XM)},
\]
where we denote $\prodod_i\U_i$ as shorthand for
$\U_n\otimes\U_{n-1}\otimes\dotsb\otimes\U_{1}$.

We let $\index_{n_1, n_2, \ldots, n_k}(I)$ be a map from a sequence of indices $I = {i_1, i_2, \ldots,
i_k}$ to an integer such that
\begin{align}
[\vect{\XM}]_{\index_{n_1, n_2, \ldots, n_k}({i_1, i_2, \ldots, i_k})} =
\XM_{i_1, i_2, \ldots, i_k}
\textrm{.}
\end{align}
We note that $\index^{-1}_N(n)$ is well-defined.

The unfold operation maps a tensor to a tensor of lower order and is defined by
\[\fold^{-1}_J: \FB^{n_1\times n_2\times \ldots \times n_k}\mapsto
\mathbb{F}^{m_1\times m_2\times \ldots \times m_l}\;,\] where $J$ is an index
set grouping of the indices $I=\{1, 2, \ldots, k\}$ into sets $J=\{J_1, J_2, \ldots,
J_l\}$, $m_o=\prod_{t\in J_o} t$, and satisfies:
\[
\vect[\fold^{-1}_J(\AM)] = \vect(\AM) \text{.}
\]
When unfolding a single index, i.e.,
$J=\{\{j\}, I-\{j\}\}$, we also denote $\fold^{-1}_J$ as $\fold^{-1}_j$.

The $\ell_p$-norm of tensor $\AM$ is defined as \[\|\AM\|_p \eqdef
\|\fold^{-1}_i \mathcal A\|_p\] for an arbitrary mode $i$.
For tensors $\AM, \BM\in \FB^{n_1\times n_2\times \ldots \times n_k}$,
$\langle \AM, \BM \rangle$ is their Frobenius inner product
defined as:
\[
\langle \AM, \BM \rangle\eqdef\langle \vect(\AM), \vect(\BM) \rangle \text{.}
\]

Finally, given $f: \FB\rightarrow\FB$ and $\T\in\FB^{n_1\times n_2\times \ldots
\times n_k}$, $f(\T)$ is defined as a tensor-valued function with $f$ applied to
each entry of $\T$. Therefore, $f(\T)\in\FB^{n_1\times n_2\times \ldots \times
n_k}$. When $f(x) = |x|$, we denote $f(\T)$ as $|\T|$.


\subsection{Group notation}

$\OG$ is the orthogonal group over real field $\RB$. $\SOG$ is the special
orthogonal group over $\RB$.
$\UG$ is the unitary group over complex field. We let
$\UUTG(n)$ denote the upper-unit-triangular group and $\LUTG(n)$ denote
the lower-unit-triangular group, both of which have all entries along the
diagonals being $1$.
$\HG$ is the group formed by (calibrated) homography transform below:
\[
\H_{2w} = \R_{2w}(\I_3+ \p_2 \frac{\n^\top}{d}) \text{,}
\]
where $\R_{2w}\in\SOG$ is attitude of the camera; $\p_2$
is position of the camera, and $\n^\top \x = d$ is equation of the object
plane.

\section{Preliminaries}
\label{sec:prelim}

In this paper we would like to show that matrix and tensor decompositions
techniques can be induced from formulations of the group orbit optimization.
As we have seen in formula (\ref{eq:goo}), a GOO problem includes two key
ingredients:
a cost function $\phi$ and a group structure $\GM$.
Thus, we present preliminaries, including \emph{sparsifying function}
and a \emph{unit matrix group}. The sparsifying functions will be used to define
cost functions for some matrix decompositions in Table~\ref{tab:matrix-decomp}
that have diagonal matrices in decomposed formulations.

It should be noted that other classes of functions can be used together with
some unit matrix groups to induce interesting matrix and tensor decompositions.
Confer Schur decomposition in Table~\ref{tab:matrix-decomp} for an example.

\subsection{Sparsifying functions}
 For two functions $f$ and $g$, we here and later denote their composition as $f\circ g$ s.t.\
$f\circ g(x)\eqdef f(g(x))$.
We first prove several utility lemmas used for characterizing sparsifying
functions.

\begin{lemma}[Subadditive properties] \label{lem:41}
If $f(\sqrt{x})$ is subadditive, then
\begin{enumerate}
\item[\emph{(1)}] $\sum_{i=1}^n f(|x_i|)\geq f(\|\x\|)$ where
$\x\in\FB^n$.
\item[\emph{(2)}] $f(0)\geq 0$.
\end{enumerate}
\end{lemma}
\begin{proof} First we have that
 $\sum_{i=1}^n f(|x_i|) = \sum_{i=1}^n f(\sqrt{x_i^2}) \geq
  f(\sqrt{\sum_{i=1}^n x_i^2}) = f(\|\x\|)$.
 By the subadditivity of $f(\sqrt{x})$ we further have $f(\sqrt{0}) +
 f(\sqrt{0}) \ge f(\sqrt{0 + 0}) = f(\sqrt 0)$, hence $f(0) = f(\sqrt 0) \ge 0$.
\end{proof}

\begin{lemma} \label{lem:42}
If $f(e^x)$ is convex for any $x \in \FB$, then when $\forall i, x_i\neq 0$, we
have: \[ \sum_{i=1}^n f(|x_i|) \geq n f\big((\prod_{i=1}^n |x_i|)^{\frac 1
n}\big)\quad .
\]
\end{lemma}
\begin{proof}
Since $f(e^x)$ is convex,  we have
\[
\sum_{i=1}^n f(|x_i|) = \sum_{i=1}^n f(e^{\ln |x_i|})\geq n f(e^{\frac{1}{n} \sum_{i=1}^n \ln |x_i|})
= n f((\prod_{i=1}^n |x_i|)^{\frac{1}{n}}).
\]
\end{proof}

\begin{lemma}\label{lem:43} If  $f$ is strictly concave and $f(0)\ge 0$, then
$f(tx)\geq t f(x)$ where $0\leq t \leq 1$, with equality only when
$t=0,1$ or $x=0$.
\end{lemma}

\begin{proof} We have $f(tx) = f(tx +(1-t)0)\geq t f(x)+(1-t)f(0)\geq t f(x)$.
Obviously,  the first equality holds only when $x=0$ or $t=0,1$.
\end{proof}

\begin{lemma}
\label{lem:44}
Assume $f(x) = f(|x|)$. Then $f$ is
concave and $f(0)\geq  0$ iff  $f$ is concave and subadditive.
\end{lemma}

\begin{proof} Because $f(x) = f(|x|)$, w.l.o.g.\ we assume $x \ge 0$.
We first prove ``$\Rightarrow$ part". When $a=0$ and $b=0$, we trivially have
$f(a)+f(b)\ge f(a+b)$. Otherwise, we have \[f(tx) = f(tx
+(1-t)0)\geq t f(x)+(1-t)f(0)\geq t f(x)\text{.}\] Thus, when $a\neq 0$ or $b
\neq 0$, \[f(a)+f(b)=f(\frac{(a+b)
a}{a+b})+f(\frac{(a+b) b}{a+b})\ge \frac{a}{a+b} f(a+b)+\frac{b}{a+b}
f(a+b)=f(a+b)\text{.}\]

As for ``$\Leftarrow$ part",  we have $f(0) + f(0)\geq f(0 + 0)$.
Hence $f(0)\geq 0$.
\end{proof}

Now we are ready to define the sparsifying function.
\begin{definition}[sparsifying function] \label{def:42}
A function f is sparsifying if
\begin{enumerate}
  \item[\emph{(a)}] $f$ is symmetric about the origin; i.e., $f(x)=f(|x|)$;
  \item[\emph{(b)}] $f(\sum_i |x_i|)=\sum_i f(|x_i|) \Longrightarrow$ there is
  at most one $i$ with $x_i\neq 0$.
\end{enumerate}
\end{definition}

The following theorem gives a  sufficient condition for function $f$  to be sparsifying.

\begin{theorem}[sufficient condition for sparsifying]\label{thm:sparsifying}
If $f(x) = f(|x|)$ and $f$ is strictly concave and subadditive, then $f$ is
sparsifying.
\end{theorem}
\begin{proof}
Because $f(x) = f(|x|)$, w.l.o.g.\ we assume $x\ge 0$. By Lemma~\ref{lem:44}, $f$ is strictly concave and $f(0)\ge 0$.
When $\sum_i x_i= 0$, there is no $i$ with $x_i=0$. Otherwise, it follows from
Lemma~\ref{lem:43} that
\[
\sum_i f(x_i) = \sum_i f \big(\frac{x_i}{\sum_{j}
x_j} \sum_j x_j \big) \geq \sum_i \frac{x_i}{\sum_{j} x_j}  f(\sum_j
x_j)=f(\sum_j x_j)\text{.}
\]
Also by Lemma~\ref{lem:43}, the equality holds iff  $\frac{x_i}{\sum_j x_j}=0
\text{ or } 1$. Because $\sum_i \frac{x_i}{\sum_j x_j} = 1$, there is only
one $i$ with $x_i\neq 0$.
In both cases, there is at most one $i$ with $x_i\neq 0$.
\end{proof}

\begin{corollary} Conical combination of sparsifying functions. In particular,
if $f$ and $g$ are sparsifying, then so is $\alpha f + \beta g$ where $\alpha$ and  $\beta$ are two nonnegative constants.
\end{corollary}
\begin{proof}
As strict concavity is preserved by conical combination, we only need
prove subadditivity is preserved by conical combination, which holds because:
\begin{align*}
(\alpha f + \beta g)(x + y) &= \alpha f(x + y) + \beta g(x + y) \\
&\le \alpha f(x) + \alpha f(y) + \beta g(x) + \beta g(y) \\
&= (\alpha f +
\beta g)(x) + (\alpha f + \beta g)(y).
\end{align*}
\end{proof}

It can be directly checked that the following functions are sparsifying.

\begin{example}
 Following functions are sparsifying:
 \begin{enumerate}
\item[\emph{(1)}] Power function: $f(x)=|x|^p$ for $0<p<1$;
\item[\emph{(2)}] Capped power function:  $f(x)=\min(|x|^p,1)$ for $0<p<1$;
\item[\emph{(3)}]  $f(x)=-|x|^p$ for  $p>1$;
\item[\emph{(4)}] $f(x)= \log(1+|x|)$;
\item[\emph{(5)}] Shannon Entropy: $f(x)=-|x| \log |x|$ when $0\leq x$;
\item[\emph{(6)}] Squared entropy: $f(x)=-x^2 \log x^2$ when $0\leq x$;
\item[\emph{(7)}] $f(x)=a-(a+|x|^p)^{\frac1p}$ for $p>1$ and  $a\geq  0$;
\item[\emph{(8)}] $f(x)=-a+(a+|x|^p)^{\frac1p}$ for $p<1$ and  $a\geq  0$;
\end{enumerate}
\end{example}

\begin{remark}
We note that $\log|x|$ is not subadditive because $f(0)=-\infty<0$.
Although $f(x)=|x|^p$ for $p<0$ is subadditive, $|x|^{p}$ is not
concave. Thus, these two functions are not sparsifying.
\end{remark}

\subsection{Unit Matrix Groups}

\begin{definition}[unit group] \label{def:unit-group}
A matrix group $\GM$ is a \emph{unit group} if $|\det(\G)|=1,\, \forall \G \in
\GM$.
\end{definition}

Clearly,  unitary, orthogonal, and unit-triangular matrix groups are unit
groups.
We now present some properties of the unit groups.

\begin{lemma}\label{lem:unit-group} Unit group has the following properties.
\begin{enumerate}
\item[\emph{(i)}] Unit group is well-defined, i.e., closed under multiplication
and inverse, and has an identity element which happens to be $\I$.
\item[\emph{(ii)}] The Kronecker product of unit groups is also a unit group. In particular, if
$\GM_1$ and $\GM_2$ are unit groups, then $\GM_1\otimes \GM_2 =\{\M_1\otimes \M_2: \M_1 \in
\GM_1 \mbox{  }  \M_2\in \GM_2 \}$ is also a unit group.
\item[\emph{(iii)}]  $\{\BP\otimes \BP^{-\top}: \BP \in \GLG(n)\}$ is a unit
group.
\item[\emph{(iv)}] $\{\A\otimes \conj{\A}: \A \in \GLG(n)\}$ is a
group, and is a unit group iff $\GLG(n)$ is a unit group.
\item[\emph{(v)}] $\{\I_n\otimes \A, \A\in\GM\}$ is a unit group iff $\GM$ is
a unit group. $\{\A\otimes\I_n, \A\in\GM\}$ is a unit group iff $\GM$ is
a unit group.
\end{enumerate}
\end{lemma}

\begin{proof}
\begin{enumerate}
  \item [{(i)}]  Let
$\G_1,\G_2\in\GM$. Then $|\det(\G_1^{-1})| = 1$ and \[ |\det( \G_1 \G_2 )| = |\det( \G_1
)| |\det(\G_2)| = 1.
\]Hence $\G_1^{-1}\in\GM$ and $\G_1\G_2\in\GM$.
  \item [{(ii)}] We first check $\GM_1\otimes \GM_2$ is a group. This
  can be done by noting that  $(\G_1\otimes\G_2)^{-1} = \G_1^{-1}\otimes\G_2^{-1}\in\GM_1\otimes \GM_2$
  when $\G_1\in\GM_1$, $\G_2\in\GM_2$; and \[(\G_1\otimes\G_2)
  (\G_3\otimes\G_4) = (\G_1\G_3)\otimes(\G_2\G_4)\text{.}\] Also $\I\in
  \GM_1\otimes \GM2$. Moreover, since
  $|\det(\G_1\otimes \G_2)| = |\det(\G_1)|^m |\det(\G_2)|^n = 1$ for any $\G_1
  \in \GM_1$ and $\G_2 \in \GM_2$, $\GM_1\otimes \GM_2$ is a unit group.
  \item [{(iii)}] Closedness under multiplication and inverse can be
  proved by noting
\[(\BP\otimes \BP^{-\top})(\Q\otimes \Q^{-\top})=(\BP \Q)\otimes (\BP^{-\top}
\Q^{-\top}) =(\BP \Q)\otimes (\BP \Q)^{-\top}\text{.}\] Also we have \[(\BP
\otimes \BP^{-\top})^{-1}=\BP^{-1}\otimes \BP^{\top}\text{.}\]  Thus $\BP
\otimes \BP^{-\top}$ forms a group with $\I$ as the identity. It is also a unit
group as $|\det(\BP\otimes \BP^{-\top})|=|\det(\BP)^n \det(\BP^{-\top})^n|=1$.
 \item[{(iv)}] Closedness under multiplication and inverse can be
  proved based on \[(\L\otimes \conj{\L})(\R\otimes \conj{\R})=(\L \R) \otimes
 (\conj{\L} \conj{\R})=(\L \R) \otimes \conj{\L \R}\text{,}\] and $(\L\otimes
 \conj{\L})^{-1}=\L^{-1} \otimes \conj{(\L^{-1})}$. Thus $\L\otimes \conj{\L}$
 forms a group with $\I$ as the identity.
 Moreover $|\det(\L\otimes \conj{\L})|=|\det(\L)^n
 \det(\conj{\L})^n|=|\det(\L)|^{2n}$, i.e., $\L\otimes \conj{\L}$ forms a unit group iff $\L$ is from a unit group.
\item[{(v)}] Note $\{\I_n\}$ is a unit group with single element. By property
{(ii)} we can prove this property.
\end{enumerate}
\end{proof}

It is worth pointing out that $\BP\otimes \BP^{-1}$ does not form a group in
general because $(\BP\otimes \BP^{-1})(\Q\otimes \Q^{-1}) = (\BP\Q)\otimes
(\Q\BP)^{-1} \not\equiv (\BP\Q)\otimes
(\BP\Q)^{-1}$.

Finally, in Table~\ref{tab:matrix-decomp} we list matrix decompositions of $\X$ used in
this paper. When referring to the Cholesky decomposition, $\X$ should be positive definite. 

\begin{table}[!ht] \centering \small
\caption{Matrix decompositions}
\begin{center}
\begin{tabular}{p{2cm} p{2.5cm} p{6cm}}
    \toprule
     Name & Decomposition & Constraint \\
    \midrule real SVD & $\X = \U \D \V^\top$ & $\U, \V\in \OG(n)$, $\D$ is
    diagonal\\
    \midrule complex SVD &  $\X = \U \D \V^*$ & $\U,\V \in \UG(n)$, $\D$ is
    diagonal\\
    \midrule QR & $\X = \Q \D \R$ & $\Q\in\UG(n), \R\in\UUTG(n)$, $\D$ is
    diagonal\\
    \midrule LU & $\X = \L \D \U$ & $\L\in\LUTG(n), \U\in\UUTG(n)$, $\D$ is
    diagonal \\
    \midrule Cholesky & $\X = \L \D \L^\top$ & $\L\in\LUTG(n)$,
    $\D$ is diagonal \\
    \midrule Schur & $\X = \Q \U\Q^*$ & $\Q\in\UG(n)$, $\U$ is upper triangular
    \\
\hline
  \end{tabular}
\end{center} \label{tab:matrix-decomp}
\end{table}

\section{Group Orbit Optimization}
\label{sec:main}

\subsection{Matrix Decomposition Induced from Group Orbit Optimization}
\label{subsec:induced-matrix-decomp}
\subsubsection{GOO formulation}
We now illustrate how matrix decomposition can be induced from GOO.
Given two groups $\GM1, \GM2$ and a data matrix
$\M$, we consider the following optimization problem
\begin{align}
\inf_{\G_1 \in \GM_1, \G_2 \in \GM_2} \; \phi(\G_2 \M\G_1^\top).
\label{eq:matrix-induce}
\end{align}

Assume that $\hat{\G_1}$ and $\hat{\G_2}$ are minimizers of the above
GOO and $\D = \hat{\G}_2 \M\hat{\G}_1^\top$,
then we refer to
\[\M = \hat{\G}_2^{-1} \D \hat{\G}_1^{-\top}\text{,}\] 
as a matrix decomposition of $\M$ which is
induced from Formula~(\ref{eq:matrix-induce}).

When $\phi = \varphi \circ \vect$, an equivalent formulation of
Formula~(\ref{eq:matrix-induce}) is:
\[
\inf_{\G_1 \in \GM_1, \G_2 \in \GM_2} \; \phi(\G_2
\M\G_1^\top)\equiv\inf_{\G\in\GM} \varphi(\G \vect(\M))\text{,}
\]
where $\G=\G_1\otimes \G_2\in\GM$ and $\GM\eqdef\GM_1\otimes\GM_2$.

\subsubsection{GOO over unit group}

For a general matrix group $\GM$, $\G\in\GM$ implies that $|\det(\G)| > 0$.
However, group structure may not be sufficient to induce non-trivial matrix
decomposition, as with some groups and cost functions the infimum will be
trivially zero. For example, with general linear group $\GLG$ and for any matrix
$\M$, we have
\[
\inf_{\G\in\GLG} \|\G \M\|_p = 0,
\]
because $s\I \in \GLG$ and
\[\lim_{s\to 0} \inf_{s\in\RB} \|s \I \M\|_p = \lim_{s\to0}s\|\M\|_p =0\;.
\]
Nevertheless, if we require $\GM$ to be a unit group, we have $|\det(\G)| = 1$.
Consequently, we can prevent the infimum from vanishing trivially for any
$\ell_p$-norm.
Thus, we mainly consider the case where $\GM$ is a unit group in this paper.

The following theorem shows that many matrix decompositions can be induced from the group orbit optimization.
\begin{theorem} 
\label{theorem:decompostion-as-optimization}
SVD, LU, QR, Schur and
Cholesky decompositions of matrix $\M \in
\mathbb{F}^{m\times n}$ can be induced from  GOO of the form
\[\inf_{\G_1 \in \GM_1, \G_2 \in \GM_2} \; \phi(\G_2 \M\G_1^\top)\text{,}
\]
by using the corresponding unit group $\G$ and
cost function $\phi$, which are given in Table~\ref{tab:group}.
\end{theorem}

Clearly, the matrix groups in Table~\ref{tab:group} are unit groups by
Lemma~\ref{lem:unit-group}.
We will prove the rest of theorem in Section~\ref{subsec:diag-goo} and
Section~\ref{subsec:triang-goo}.

\begin{remark}
The cost function
for SVD, QR and Matrix Equivalence can be $\phi(\X) = \|\X\|_p, \, 1\le p <2$. And
the cost function for LU, Schur and Cholesky can be $\phi(\X) =
\sum_{ij}\|x_{ij} {\mathbb I}_{\{i<j\}}\|_p, \, 1\le p <2$.
\end{remark}

\begin{table}[!ht] \centering \small
\caption{Matrix decompositions induced from optimizations}
\begin{center}
\begin{tabular}{p{3.8cm} p{3.8cm} p{4cm}}
	\toprule
     Decomposition & Unit group $\GM=\GM_1\otimes\GM_2$ & Objective function
     $\phi(\X)$
     \\
	\midrule real SVD: $\U \D \V^\top$ & $\{\V\otimes \U: \U, \V\in \OG(n)\}$
    & $\sum_{ij} f(|x_{ij}|)$ where $f(\sqrt{x})$ is strictly
    concave, $f(0)\ge0$\\
    \midrule complex SVD:  $\U \D \V^*$ & $\{\conj{\V}\otimes \U: \U,\V
    \in \UG(n)\}$ & $\sum_{ij} f(|x_{ij}|)$ where $f(\sqrt{x})$ is strictly
    concave, $f(0)\ge0$\\
    \midrule QR: $\Q \D \R$ & $\{\R^\top\otimes \Q: \Q\in\UG(n),
    \R\in\UUTG(n)\}$ & $\sum_{ij} f(|x_{ij}|)$ where $f(\sqrt{x})$ is strictly
    concave and increasing, $f(0)\ge0$\\
    \midrule Matrix Equivalence: $\BP \D \Q$ & $\{\Q^\top\otimes \BP:
    \Q, \BP\in\SLG(n)\}$ & $\sum_{ij} f(|x_{ij}|)$ where
    $f(0)\ge0$, $f(\sqrt{x})$ is strictly concave and increasing;
    $f(\sqrt{e^x})$ is convex\\
    \midrule LU: $\L \D \U$ & $\{\I\otimes \L:
    \L\in\LUTG(n)\}$ & $\sum_{ij} f(|x_{ij} {\mathbb I}_{\{i<j\}} |)$ where
    $f(x)\not\equiv 0$, $f(x)=0\Rightarrow x=0$, $f(x)\ge 0$\\
    \midrule Cholesky: $\L \D \L^\top$ & $\{\I\otimes \L:
    \L\in\LUTG(n)\}$ & $\sum_{ij} f(|x_{ij} {\mathbb I}_{\{i<j\}} |)$ where
    $f(x)\not\equiv 0$, $f(x)=0\Rightarrow x=0$, $f(x)\ge 0$\\
    \midrule Schur: $\Q \U\Q^*$ & $\{\conj\Q\otimes \Q:
    \Q\in\UG(n)\}$ & $\sum_{ij} f(|x_{ij} {\mathbb I}_{\{i<j\}} |)$ where
    $f(x)\not\equiv 0$, $f(x)=0\Rightarrow x=0$, $f(x)\ge 0$\\
  \end{tabular}
\end{center} \label{tab:group}
\end{table}

\begin{remark}
The formulation of QR decomposition exploits
the fact that $\M=\Q \R$ is
equivalent to $\M=\Q(\D \tilde{\R})$ where $\Q\in\UG(n)$, $\R$ is
upper-triangular, $\tilde{\R}\in\UUTG(n)$, and $\D$ is diagonal.
\end{remark}

\begin{remark}
``Matrix Equivalence" in Table~\ref{tab:group} finds
a diagonal matrix equivalent to an invertible matrix $\M$ as defined in
Section~\ref{subsub:equiv}.
\end{remark}

\begin{remark}
However, there are matrix decompositions whose formulation cannot be
expressed as GOO in the same way as Table~\ref{tab:group}. For example,  Polar
decomposition $\M=\U \L \D \L^*$ where $\U\in\UG(n)$ and $\L\in\LUTG(n)$, though
derivable from SVD, cannot be induced from a GOO formulation of
diagonalization.
This is because $\conj{\L}\otimes \U \L$ does not form a group as it is not
closed under multiplication.
For another example, consider a formulation of decomposition $\M =\L \D
\L^{-\top}$ where $\L\in\LUTG(n)$ and $\D$ is diagonal. As we stated earlier, $\L\otimes
\L^{-1}$ is not a group in general, so  $\S =\L \D \L^{-\top}$ cannot be induced
from a GOO formulation of diagonalization.
\end{remark}

\begin{remark}
For matrix decomposition of the form $\M =\A \D \B^\top$, where  $\A \in \mathbb{F}^{m\times r}$ and $\B \in
\mathbb{F}^{n\times r}$ with  $r\leq \min(m, n)$.
In this case, we can zero-pad $\D$ to $\tilde{\D}\in\mathbb{F}^{m\times n}$, and
extend $\A$ and $\B$ to $\tilde{\A}\in\mathbb{F}^{m\times m}$ and
$\tilde{\B}\in\mathbb{F}^{n\times n}$ which are square matrices. Accordingly, we formulate a decomposition $\M =\tilde{\A}\tilde{\D}\tilde{\B}^\top$ which may be induced from GOO.
\end{remark}

We next prove a lemma that characterizes the optimum.
\begin{lemma}[Criteria for infimum] \label{lem:criteria-infimum}
If $\phi(\G \D) \geq \phi(\D)$ for any $\G\in\GM$ and there exists $\A \in
\GM$ s.t.\ $\M= \A \D$, then
\[
\inf_{\G\in\GM} \phi(\G \M) = \phi(\D).
\]
\end{lemma}

\begin{proof}
We note that $\inf_{\G \in \GM} \phi(\G \M) = \inf_{\G \in \GM} \phi(\G \A
\D)$. By the group structure, the coset $\{\G \A: \G\in\GM\} =
\GM$. Hence we have
\[
\inf_{\G\in\GM} \phi(\G \M) = \inf_{\G \in \GM} \phi( \G \A \D) =
\inf_{\G\in\GM} \phi(\G \D).
\]
Using the condition $\forall \G\in\GM, \phi(\G\D) \geq \phi(\D)$, we
have \[ \inf_{\G\in\GM} \phi(\G \D) \geq
\inf_{\G\in\GM} \phi(\D)=\phi(\D).
\]
On the other hand, as $\I\in\GM$ we have $\phi(\D)=\phi(\I\D)\geq
\inf_{\G\in\GM} \phi(\G \D)$.
Hence \[\phi(\D) = \inf_{\G\in\GM} \phi(\G \D)= \inf_{\G\in\GM} \phi(\G
\M)\text{.}\]
\end{proof}

By virtue of Lemma~\ref{lem:criteria-infimum}, if we want to prove that matrix decomposition
$\vect{(\M)} =\tilde{\G}\vect{(\D)}$ is induced by a GOO \wrt $\phi$ and $\GM$, we only need prove
that there exists a $\tilde\G\in\GM$ s.t.\ $\vect{(\M)} =\tilde{\G}\vect{(\D)}$, and
$\phi(\G \vect(\D)) \geq \phi(\vect(\D))$ $\forall \G\in\GM$.
The equality condition  will determine the uniqueness of the
optimum of the optimization problem.

\subsection{Matrix Diagonalization as GOO}\label{subsec:diag-goo}

Next we demonstrate how matrix diagonalization can be induced from GOO with
proper choice of cost function and unit group.
\subsubsection{Singular Value Decomposition}

First we discuss   SVD of a complex matrix and of a real matrix.
\begin{lemma}[Cost function and group for SVD]\label{lem:svd-cost} Let
$\D=[d_{ij}]$ be pseudo-diagonal, and  $\U,\V \in\UG(n)$.
Given a function $f$ such that $f(x) = f(|x|)$ and $f(\sqrt{|x|})$ is
strictly concave and subadditive, and $\phi(\X) = \sum_{ij} f(x_{ij})$ we have \[
\phi(\U\D\V^* )\geq \phi(\D), \] 
with equality  iff there exists a row
permutation matrix $\BP$ such that $|\U\D\V^*| = |\BP\D|$.

Furthermore, if $\U, \V \in\OG(n)$, we have
\begin{align}\label{ineq:svd}
\phi(\U\D\V^\top)\geq
\phi(\D),
\end{align}
with equality  iff there exists a row permutation
matrix $\BP$ such that $|\U\D\V^\top| = |\BP\D|$.
\end{lemma}

\begin{proof}
First we prove the inequality.
We write  $g(x)=f(\sqrt{x})$ and $\A =
\U\D\V^*$. We  let $g(\X) = [g(x_{ij})]$ be a matrix-valued function of $\X=[x_{ij}]$. As $g$ is
concave and subadditive, by Lemma~\ref{lem:41} for a vector $\v=(v_1, \ldots,
v_n)^\top$, we have $\sum_{i=1}^n f(v_i) \ge f(\norm\v)= g(\v^*\v)$. Applying
this to each column  of $\A$, we have
\begin{align}\label{eq:subadditive}
\phi(\A)\geq \tr[ g (\A^*\A)] = \tr[g(\V\D^*\D\V^*)].
\end{align}
Alternatively, we can also apply the inequality to each row  of $\A$ and have
\begin{align}\label{eq:subadditive-2}
\phi(\A)\geq \tr[ g (\A \A^*)] = \tr[g(\U\D\D^*\U^*)].
\end{align}
As $\D$ is pseudo-diagonal, $\D^*\D$ is diagonal. Because $g$ is concave and $\V \V^*=\I$, we can apply Jensen's
inequality, obtaining \[
\tr( g(\V\D^*\D \V^*))\geq \tr (\V (g  (\D^*\D))\V^*).
\]
Hence altogether we have:
\[
\phi(\A)\ge\tr (g (\V\D^*\D \V^*))\geq \tr
(\V (g (\D^*\D))\V^*) = \tr (g(\D^*\D)) = \sum_{ij} f(d_{ij}) =\phi(\D).
\]

Next we check the equality condition.
By Theorem~\ref{thm:sparsifying}, $g$ is sparsifying. For the
equality condition in inequality~(\ref{eq:subadditive}) to hold, $\A$ can have
at most one nonzero in each column. By the symmetry between (\ref{eq:subadditive}) and
(\ref{eq:subadditive-2}), and noting $\phi(\A^\top) = \phi(\A)$ and $
\phi(\D) = \phi(\D^\top)$, $\A$  can also have at most one nonzero in each row for
$\phi(\A) = \phi(\D)$ to hold.
Hence when the equality holds, $\A$ is pseudo-diagonal. Then there exists a
permutation matrix $\BP$ such that $\Z = \BP^{-1}|\A| = \BP^{-1} \Q \A$ is a
diagonal matrix with elements on diagonal in descending order and are all non-negative,
where $\Q$ is a diagonal matrix s.t.\ $|\Q| = \I$.
By the uniqueness of singular values of a matrix, we have $\Z = |\D|$. Hence
equality in inequality\ref{ineq:svd} holds when $|\A| = \BP|\D| = |\BP\D|$.

The proof for $\U, \V \in\OG(n)$ is similar.
\end{proof}

Note that $\D$, modulo sign and permutation, is the global
minimizer for a large class of functions $f$.

After applying Lemma~\ref{lem:criteria-infimum}, we have the following theorem.
\begin{theorem}[SVD induced from optimization]
\label{thm:svd-opt}
We are given a function $f$ such that
$f(x) = f(|x|)$ and $f(\sqrt{|x|})$ is strictly concave and subadditive, and
$\phi(\X) = \sum_{ij} f(x_{ij})$.
Let $\hat{\U}$ and $\hat{\V}$ be an optimal solution of the following optimization:
\[
\inf_{\U\in\UG(n),\V\in\UG(n)} \phi(\U^*\M\V).
\]
Then if SVD of $\M$ is $\M=\U\S\V^*$, there exist a permutation matrix
$\BP$ and a diagonal matrix $\Z$ such that $\M = \hat{\U}\BP
\Z\hat{\V}^*$ and $|\Z|=\S$.
\end{theorem}

\begin{corollary}
With $\phi$ as in Theorem~\ref{thm:svd-opt}, eignedecomposition of a Hermitian
matrix $\M$ can be induced from
\[
\inf_{\U\in\UG(n)} \phi(\U^*\M\U).
\]

Similarly, eignedecomposition of a real symmetric matrix $\M$ can be induced from
\[
\inf_{\U\in\OG(n)} \phi(\U^\top\M\U).
\]
\end{corollary}

From the above optimization, we can derive several inequalities.

\begin{corollary}[The Schatten $p$-norm and $\ell_p$-norm inequality]
\label{cor:schatten-ineq}
The
$\ell_p$-norm of matrix $\A$ is larger (smaller) than the Schatten $p$-norm of
$\A$ when $0\leq p< 2$ $(>2)$.

In particular, we have
\[
\|\M\|_p \ge \inf_{\U,\V\in\UG} \|\U\M\V^*\|_p = \|\M\|_{*p} \quad \text{ when }
0<p< 2,
\]
and
\[
\|\M\|_p \le \sup_{\U,\V\in\UG} \|\U\M\V^*\|_p = \|\M\|_{*p}\quad \text{ when }
p> 2.
\]
\end{corollary}

\begin{proof} $f(x)=|x|^p$ satisfies $f(0)\geq 0$, and $f(\sqrt{|x|})$ is
strictly concave when $0\leq p <2$. On the other hand, $f(x)=-|x|^p$ satisfies
$f(0) \geq 0$ and $f(\sqrt{|x|})$ is strictly concave when $p>2$. By
Theorem~\ref{thm:svd-opt} we have
\[
\inf_{\U,\V\in\UG} \|\U\M\V^*\|_p = \|\M\|_{*p} \quad \text{ when }
0<p< 2,
\]
and
\[
\inf_{\U,\V\in\UG} -\|\U\M\V^*\|_p = -\|\M\|_{*p}\quad \text{ when }
p> 2.
\]

\end{proof}

\begin{corollary} [Duality gap]
Given SVD of $\M$ as $\M=\U\D\V^*$,  we have
\[ \|\M\|_p \geq \|\D\|_p\geq \|\D \|_q\geq  \|\M\|_q  \]
where $0<p<2<q$.
\end{corollary}
\begin{proof}
Due to the non-increasing property of the $\ell_p$-norm \wrt $p$, we have
$\|\D\|_p\ge \|\D\|_q$. Also by Theorem~\ref{thm:svd-opt} we have
$\|\D\|_p = \|\M\|_{*p}$ and $\|\D\|_q
= \|\M\|_{*q}$. Applying Corollary~\ref{cor:schatten-ineq} completes the proof.
\end{proof}

\begin{corollary}
The von Neumann entropy of density matrix $\M$ is smaller than the sum of the Shannon
entropies of rows (columns) of $\M$.
\end{corollary}
\begin{proof}
By noting that $f(x) = -x\log{x}$ is strictly concave and $f(0)\geq 0$ when
$x\geq 0$, and that the von Neumann entropy is entropy of diagonal matrix in SVD
of $\M$, the inequality holds.
\end{proof}

\subsubsection{QR Decomposition}

To derive GOO for QR decomposition, we first note that QR decomposition of a
matrix $\M$ can be rewritten as $\M= \Q \D \R$, where $\Q\in\UG(n)$,
$\D$ is diagonal, and $\R\in\UUTG(m)$.

\begin{lemma}[Cost function and group for QR]
\label{lem:qr-cost}
Let $f$ satisfy that $f(x) =
f(|x|)$, $f(\sqrt{x})$ is concave and increasing; $f(0)\geq 0$. Let $\phi(\X) = \sum_{ij} f(x_{ij})$.
If $\Q\in\UG(n)$, $\R\in\UUTG(n)$, and $\D$ is diagonal, then we have
\[
\phi(\Q \D \R) \geq \phi(\D),
\]
with equality  when \ $|\Q \D \R| = |\BP\D|$, where $\BP$
is a row permutation matrix.
\end{lemma}

\begin{proof}
Let $g(x)=f(\sqrt{|x|})$ and $\A=\Q\D\R$, and let $g(\X)=[g(x_{ij})]$.
First we prove the inequality.
As $g(x)$ is sparsifying and increasing, we have
\[
\phi(\A)\geq \tr [g (\A^*\A)] = \tr [g (\R^* \D^*\D
\R)] \geq \tr [g (\D^*\D)] = \phi(\D).
\]

Next we check the equality condition.
For the equality to hold in inequality $\tr[ g (\R^* \D^*\D
\R)] \geq \tr[g (\D^*\D)]$, as $g$ is increasing, $\R$ needs to be diagonal. Hence, $\R
= \I$.
Now for the equality to hold in $\phi(\A)\ge \tr [g (\A^*\A)]$, $\A$ needs to be
pseudo-diagonal.
Thus, the equality holds only when $|\A| = \BP|\D| = |\BP \D|$ where $\BP$ is a
row permutation matrix.
\end{proof}

Similarly, we derive the optimization inducing QR decomposition.
\begin{theorem}[QR induced from optimization]
Assume that the conditions are satisfied in Lemma~\ref{lem:qr-cost}.
Let $(\hat{\Q}, \hat{\R})$ be optimal solution of optimization as follows
\[
\inf_{\Q\in\UG(n),\R\in\UUTG(n)} \phi(\Q^{-1}\M\R^{-1})\textrm{.}
\]
Then $\M =
\hat{\Q} \Z$ is the QR decomposition of $\M$, where
$\Z = \hat{\Q}^{-1}\M$.
\end{theorem}

\subsubsection{Matrix Equivalence by the Special Linear
Group}\label{subsub:equiv}

An interesting question is whether we can extend  the following optimization form
to more general groups $\GM_1$ and $\GM_2$:
\[
\inf_{\U\in\GM_1, \; \V\in\GM_2} \; \phi(\U\M\V).
\]
It turns out that we can use the special linear group to construct a unit group, and
hence, induce matrix equivalence from an optimization.

\begin{lemma}[Matrix equivalence decomposition]
An  invertible matrix $\M \in \mathbb{F}^{n\times n}$ can be decomposed as
$\M=\A \D \B$, where $\det(\A)=\det(\B)=1$, and $\D=\det(\M)^{1/n} \I$.
\end{lemma}
\begin{proof}
Just let
\begin{align}
\A=\det(\M)^{-1/n} \M, \quad \B=\I_n, \quad \D =\det(\M)^{1/n} \I,
\end{align}
which gives an existence proof.
\end{proof}

\begin{lemma}
\label{lem:matrix-equivalence-cost}
If $f(\sqrt{x})$ is strictly concave and increasing, and $f(\sqrt{e^x})$ is
convex, $f(0)\geq 0$, $\A,\B\in\SLG(n)$ and $\D = \lambda\I$, then $\phi(\A \D
\B)\geq  \phi(\D)$, with equality  iff there exists a permutation matrix $\BP$
such that $|\A \D \B| = |\BP \D|$.
\end{lemma}
\begin{proof}
We write   $g(x)=f(\sqrt{x})$ for shorthand.
First prove the inequality.
As $\D = \lambda\I_n$, we have
\[ \phi(\A \D \B) = \phi(\A\B \D).
\]
As $f(\sqrt{x})$ is strictly concave and $f(0)\geq 0$, we have
\[\phi(\A\B \D) \ge \tr[ g (\A\B \D\D^*\B^*\A^*)] = \tr[
g (\lambda^2\A\B\B^*\A^*)].
\]

As $\A\B\B^*\A^*$ is Hermitian, we let its LDL decomposition be $\L\Z\L^*$ where
$\L\in\LUTG(n)$. Because $g$ is increasing, we have
\[
\tr[ g (\lambda^2 \L\Z\L^*)] \ge \tr[ g (\lambda^2 \Z)].
\]
Because $g(e^x)$ is convex and
\[
\det(\Z) = \det(\L\Z\L^*) = \det(\A\B\B^*\A^*) =1\text{,}
\] we have
\[
\tr[ g (\lambda^2 \Z)] \ge \tr[ g(\lambda^2 \I)].
\]
In summary,  we have
\[
\phi(\A\D \B)\ge \tr[ g(\lambda^2 \I)] = \tr[ g(\D\D^*)] = \phi(\D)\textrm{.}
\]

Next we check the equality condition. The equality holds in \[\tr[ g(\lambda^2
\L\Z\L^*)] \ge \tr[ g(\lambda^2 \I)]\] iff $\L\Z\L^* = \I$ and $\A\B\D$ is
pseudo-diagonal. Hence, the equality holds iff $|\A \D \B| = \BP |\D| = |\BP
\D|$.
\end{proof}

\begin{theorem}[Matrix equivalence induced from optimization]
Let $f$ satisfy that $f(x) =
f(|x|)$, $f(\sqrt{x})$ is concave, $f(\sqrt{e^x})$ is convex and increasing, and
$f(0)\geq 0$.
Let $(\hat{\A}, \hat{\B})$ be an optimal solution of the following optimization problem
\[
\inf_{\A\in\SLG(n),\B\in\SLG(n)} \phi(\A^{-1}\M\B^{-1})\textrm{.}
\]
Then there exists a row permutation matrix $\BP$ such that $\M =
\hat{\A} (\BP \Z) \hat{\B}$ is the matrix equivalence decomposition of $\M$ with
$\BP \Z = \lambda\I$, where $\Z = (\hat{\A}\BP)^{-1}\M\hat{\B}^{-1}$.
\end{theorem}


\subsection{Matrix Triangularization as GOO}\label{subsec:triang-goo}

Next we demonstrate how matrix triangularization can be induced from GOO with
proper choice of cost function and unit group. In fact, we can prove that any
triangularization can be induced from optimization \wrt a masked norm.
\begin{lemma}
A matrix decomposition $\M = \G_2 \U \G_1^\top, \G_1\in\GM_1, \G_2\in\GM_2$,
where $\U$ is upper triangular and $\GM_1 \otimes \GM_2$ is a unit group, can
be induced from the following optimization:
\begin{align}\label{form:triang}
\inf_{\G_1\in\GM_1, \G_2\in\GM_2} \phi(\G_2\M\G_1^\top)
\textrm{.}
\end{align}
Here $\phi(\X) = \sum_{ij} f(|x_{ij} {\mathbb I}_{\{i<j\}} |)$ where
$f(x)=0\Rightarrow x=0$, $f(x)\not\equiv 0$, $f(x)\ge 0$.
\end{lemma}
\begin{proof}
We trivially have
\[\forall \G, \; \phi(\G_2 \U \G_1^\top) \ge 0 = \phi(\U)\text{.}
\]
By Lemma~\ref{lem:criteria-infimum} and
$\phi(\X) = 0\iff \X=\0$, the decomposition can be induced from optimization.
\end{proof}

For example, the Schur decomposition can be induced from
Formula~\ref{form:triang} with $f(x) = |x|$.
We will give a numerical example in
Section~\ref{sec:exps}.

\section{Group Orbit Optimization on Tensor Data}
\label{sec:tensor}

The GOO problem on tensor $\T$ is defined as
\[\inf_{\G_i \in \GM_i} \phi(\prod_i \T \G_i).
\]
When there exists function $\varphi$ s.t.\ $\phi(\T) = \varphi(\vect(\T))$, we
get a form that bears resemblance to the matrix version:
\[\inf_{\G_i \in \GM_i} \varphi(\vect(\prod_i \T \G_i) )= \inf_{\G_i \in \GM_i} \varphi(\prodod_i\G_i\vect(\T)).
\]

Similar to the matrix case in Section~\ref{subsec:induced-matrix-decomp}, we now
illustrate how the Tucker decomposition can be induced from an optimization formulation.
Given a group $\GM=\prodod_i\GM_i$ and $\G=\prodod_i\G_i$ and a tensor $\T$,
we define the following optimization problem
\[\inf_{\G \in \GM} \; \varphi(\G
\vect(\T)) = \inf_{\G_i \in \GM_i} \; \varphi((\prodod_i\G_i) \vect(\T))
=\inf_{\G_i \in \GM_i} \; \varphi(\vect(\prod_i\T \G_i))).
\]
If we assume that  $\hat{\G} =\prodod_i{\hat{\G}_i}= \arg \inf_{\G \in
\GM} \; \varphi(\G \vect(\T))$ and $\ZM = \prod_i \T \hat{\G}_i $, then
$\MM = \prod_i \ZM \hat{\G}_i^{-1}$ can be regarded as a tensor decomposition induced from
the optimization problem.

In this section,  we particularly generalize the results in
Sections~\ref{sec:prelim} and \ref{sec:main} to tensors. In
Lemma~\ref{lem:subgroup}, we prove that we can use the subgroup relation to
induce a partial order of the infima of GOO.
We also show that
 GOO \wrt the special linear group finds the ``sparsest" Tucker-like
 decomposition of a tensor, and prove that GOO on tensor $\T$ is ``denser" than
 GOO on any matrix unfolded from $\T$.
We also prove Theorem~\ref{thm:unfolded-diagonal}, which says that if a tensor
can be decomposed into a core tensor with certain shape, then it is optimal. As
a consequence, we prove that not all tensors have superdiagonal form under a GOO
\wrt any matrix group.

\subsection{Subgroup Hierarchy}

First we observe the following partial order of infima of GOO induced from a
subgroup relation.

\begin{lemma}[Infima partial order from subgroup relation]\label{lem:subgroup}
If $\GM_1$ is a subgroup of $\GM_2$, then for any $\phi:\FB^{n_1\times
n_2\times\cdots n_k} \rightarrow \RB$:
\[\inf_{\G \in \GM_1} \phi(\G \M) \ge  \inf_{\G \in \GM_2} \phi(\G
\M)\textrm{.}\]
\end{lemma}

\begin{proof}
As $\GM_1$ is a subgroup of $\GM_2$, the set of optimization variables of the left-hand side is
a subset of those of the right-hand side. Hence, the inequality holds.
\end{proof}

\begin{corollary}
For a matrix $\M$, we can construct an upper bound of the Schatten $p$-norm via
\[
\inf_{\G\in\UG}\|\G\M\|_p \ge \|\M\|_{*p}.
\]
\end{corollary}
\begin{proof}
\[\inf_{\G\in\UG}\|\G\M\|_p = \inf_{\G\in\UG}\|(\I\otimes\G)\vect(\M)\|_p
\ge \inf_{\G_1, \G_2\in\UG}\|(\G_2\otimes\G_1)\vect (\M) \|_p =
\|\M\|_{*p}.\]
\end{proof}


\begin{lemma}[GOO \wrt special linear group]\label{lem:special_linear}
The infimum of GOO \wrt the special linear group is the smallest among all
GOO \wrt a unit matrix group $\GM$ and the same $\phi$ for a tensor, that is, 
\begin{align}
\inf_{\prodod_i\G_i\in\GM} \phi(\prod_i \T \G_i)
\ge \inf_{\G_i \in \SLG(n_i)} \phi(\prod_i \T \G_i) \textrm{.}
\end{align}
\end{lemma}

\begin{proof}
First we note for $\A\in \mathbb{F}^{m\times m}$, $\B\in \mathbb{F}^{n\times
n}$, \[ \det(\A\otimes
\B)^{\frac{1}{mn}}=\det(\A)^{\frac1m}\det(\A)^{\frac1n} \textrm{.}\]
Let $\Z_i\in\GLG(n_i)$ and $\det(\prodod_i\Z_i)=1$. We can find
$\N_i\in\SLG(n_i)$ such that \[\N_i=(\det(\Z_i))^{-\frac1{n_i}}\Z_i\;.\]
Now as $\det(\prodod_i\Z_i)=\det(\prodod_i\N_i)=1$, we
have $\prodod_i\N_i=\prodod_i\Z_i$.
By Lemma~\ref{lem:subgroup}, we have
\[
\inf_{\prodod_i\G_i \in\GM} \phi(\prod_i \T \G_i)\ge \inf_{\Z_i\in\GLG(n_i),
\det(\prodod_i\Z_i)=1} \phi(\prod_i \T \Z_i) = \inf_{\G_i
\in \SLG(n_i)} \phi(\prod_i \T \G_i)\textrm{.}
\]

\end{proof}

Next we show that GOO gives a unified framework for matrix decomposition and
tensor decomposition. We can rewrite decomposition in Table~\ref{tab:group} in
tensor notation as \[\M =\prod_i \D \G_i \text{,}\]
which is induced \wrt a cost function $\phi$ by optimization
\[\inf_{\G\in\GM_i} \phi(\prod_i\M \G_i)\text{.}
\]
Now if there exists $\varphi$ s.t.\ $\phi = \varphi \circ \vect$,
we can generalize $\phi$ to tensor as
\[\tilde{\phi}(\T) = <\1,\, f(\T)>\text{,}
\]
where $\1$ is a tensor with all entries being $1$ and of the same dimension as
$\T$.

This inspires us to define a tensor version of the above optimization and
decomposition as below:
\[
\T=\prod_i \ZM\G_i \text{,}\]
and \[
\inf_{\G\in\GM_i} \tilde{\phi}(\prod_i\T\G_i)\text{.}
\]

In particular, if a matrix decomposition can be induced by entry-wise cost
function $\phi(\M) = \sum_{ij} f(m_{ij})$ \wrt some unit group, we can consistently
generalize the matrix decompositions to tensors using cost function
$\tilde{\phi}(\T) = <\1, f(\T)>$. In this case, there is an inequality relation
that follows from the Lemma~\ref{lem:subgroup}.

\begin{lemma}[Lifting lemma]
\label{lem:lifting}
We are given a tensor $\T$ and its arbitrary unfolding
$\fold^{-1}_I(\T)$ \wrt an index set grouping $I$.
Let $\{m_i\}$ and $\{n_j\}$ be the sizes of the square matrices before and
after grouping. Then we have
\[
\inf_{\M_i\in\GM(m_i)} \phi(\prod_i \T \M_i) \ge
\inf_{\N_j\in\GM(n_j)} \phi(\prod_j \fold^{-1}_I(\T) \N_j).
\]
\end{lemma}

\begin{proof}
We note $\GM(a)\otimes\GM(b)$ is a subgroup of $\GM(a+b)$. Hence
\begin{align*}
\inf_{\M_i\in\GM(m_i)} \phi(\prod_i \T \M_i)
&= \inf_{\M_i\in\GM(m_i)} \phi((\prodod_i \M_i)\vect (\T)) \\
&\ge
\inf_{\N_j\in\GM(n_j)} \phi((\prodod_j\N_j) \vect (\T) )\\
&= \inf_{\N_j\in\GM(n_j)} \phi(\prod_j \fold^{-1}_I(\T) \N_j).
\end{align*}

\end{proof}

\subsection{An Upper Bound for Some Tensor Norms}
In the literature, there are multiple generalizations of the Schatten $p$-norm to
tensors.
For example, the tensor unfolding
trace norm \cite{Liu:2013:TCE:2412386.2412938, tomioka2010extension} is
defined as a weighted sum of the trace norm of single index unfoldings of the tensor; namely,
\begin{align}\label{eq:tensor-unfold-norm}
\sum_{i=1}^k \alpha_i \|\fold^{-1}_i \T\|_{*} \; ,
\end{align}
where $\alpha_i\ge 0$ and $\sum_i\alpha_i = 1$.

Another generalization as given in \cite{signoretto2010nuclear} is defined by
\begin{align}\label{eq:tensor-unfold-norm-2}
(\sum_{i=1}^k \frac1k \|\fold^{-1}_i \T\|_{*q}^p)^\frac{1}{p} \; .
\end{align}

These tensor norms are interesting as they correspond to the Schatten norm
of matrix. We next study use of the Lemma~\ref{lem:lifting} to construct an upper
bound for the two norms.

Formula~\ref{eq:tensor-unfold-norm} and Formula~\ref{eq:tensor-unfold-norm-2}
try to capture the tensor structure by considering all single-index unfoldings of the tensor. However, there are many unfoldings that
are not single index. In general, for a $k$th-order tensor, there are $2^{k-1}-2$
possible unfoldings, as in the following example.
\begin{example}
A \nth{3} order tensor has 3 unfoldings \wrt the following index set grouping:
$\{\{1\},\{2,3\}\},\{\{2\},\{1,3\}\},\{\{3\},\{1,2\}\}$.
Here $\{\{1\},\{2,3\}\}$ means putting index 1 of the tensor in the first dimension
of the unfolded matrix, and index 2 and 3 of the tensor in the second dimension of the unfolded
matrix.
Additionally, a \nth{4} th-order tensor has 6 unfoldings.
\end{example}

It turns out the following GOO that respects the tensor structure produces an
upper bound for the Schatten $p$-norm of the matrices unfolded from a tensor.

\begin{lemma}[Infimum of GOO \wrt the unitary group]
\label{lem:unitary-goo-infimum}
For $0\le p<2$, and for any index set grouping $J$, we have:
\[\inf_{\G_i\in\UG_i} \|\prod_i \T \G_i\|_p \ge  \inf_{\G_j\in\UG_j} \|\prod_j
(\fold^{-1}_J \T) \G_j\|_p\; . \]

Similarly for $p>2$, we have:
\[\sup_{\G_i\in\UG_i} \|\prod_i \T \G_i\|_p \le  \sup_{\G_j\in\UG_j} \|\prod_j
(\fold^{-1}_J \T) \G_j\|_p\; . \]
\end{lemma}

\begin{proof}
The inequalities is obtained by applying Lemma~\ref{lem:lifting} to GOO \wrt
unitary group and $f(x)=\|x\|_p$ when $0\le p<2$ and $f(x)=-\|x\|_p$ when $p>2$.
\end{proof}

\begin{corollary}
For $0\le p<2$, we have
\[
\inf_{\G_i\in\UG_i} \|\prod_i \T \G_i\|_p \ge \sum_{i=1}^k \alpha_i
\|\fold^{-1}_i \T\|_{*} \; ,
\]
where $\alpha_i\ge 0$ and $\sum_i\alpha_i = 1$.
We also have:
\[
\inf_{\G_i\in\UG_i} \|\prod_i \T \G_i\|_p \ge
	(\sum_{i=1}^k \frac1k \|\fold^{-1}_i \T\|_{*p}^q)^\frac{1}{q}.
	\]
	
\end{corollary}
\begin{proof}
By Lemma~\ref{lem:unitary-goo-infimum}, we have
\[
\inf_{\G_i\in\UG_i} \|\prod_i \T \G_i\|_p \ge \max_{1\le i\le k}
\inf_{\G_j\in\UG_j} \|\prod_j (\fold^{-1}_i \T) \G_j\|_p = \max_{1\le i\le k}
\|\fold^{-1}_i \T\|_{*p}\;.
\]
We prove the inequalities as the following relations hold:
\[
\max_{1\le i\le k} \|\fold^{-1}_i \T\|_{*p} \ge \sum_{i=1}^k \alpha_i \|\fold^{-1}_i
 \T\|_{*} \;,
\]
and
\[
\max_{1\le i\le k} \|\fold^{-1}_i \T\|_{*p} =
(\sum_{i=1}^k\frac1k\max_{1\le i\le k} \|\fold^{-1}_i
\T\|_{*p}^q)^\frac{1}{q} \ge
(\sum_{i=1}^k \frac1k \|\fold^{-1}_i
\T\|_{*p}^q)^\frac{1}{q}\;.
\]
\end{proof}

Due to the above inequality, an optimization that tries to minimize one of
tensor norms defined as in
Formula~\ref{eq:tensor-unfold-norm} and Formula~\ref{eq:tensor-unfold-norm-2} can have the tensor norm replaced by
$\inf_{\G_i\in\UG_i} \|\prod_i \T \G_i\|_p$, as minimizing upper bound of a
function $f$ can always minimize $f$.

\subsection{Sparse Structure in Tensor}

The tensor rank is defined as the minimum number of non-zero rank-1 tensors required to
sum up to $\T$, which is a generalization of the matrix rank. In the Tucker
decomposition $\T=\prod_i \XM\U_i$, one can define the Tucker rank \wrt the
different constraints on the $\U_i$. For example, for a real tensor,
when the $\U_i$ are required to be orthogonal, the number of nonzeros in $\XM$
is defined as a tensor strong orthogonal rank of $\T$ \cite{Kolda:2001:OTD:587708.587830}.
Trivially, the tensor rank is a lower bound of all the Tucker ranks.

Tensor decomposition has a large body of the literature
\cite{harshman1970foundations}
\cite{Zhang:2001:RAH:587704.587760}
\cite{Lathauwer:2000:BRR:354353.354405}
\cite{Lathauwer:2000:MSV:354353.354398}
\cite{Lathauwer:2005:CCD:1039893.1039944}
\cite{DeLathauwer:2008:DHT:1461964.1461967} \cite{ishteva2008dimensionality}, the
interested reader may refer to \cite{Kolda:2009:TDA:1655228.1655230} and the references therein.

We find that the strong orthogonal rank of tensor $\T$
is exactly the infimum of the following GOO problem
\[
\sorank\T = \inf_{\G_i\in\UG_i} \nnz{\prod_i \T \G_i}.
\]

\begin{corollary}
\label{cor:special-linear-best} We have
\[
\inf_{\det(\prodod_i\G_i)=1} \nnz{\prod_i \T \G_i}
\ge \inf_{\G_i \in \SLG(n_i)} \nnz{\prod_i \T \G_i} \ge \rk(\T) \textrm{.}
\]
\end{corollary}
\begin{proof}
The first inequality directly follows from Lemma~\ref{lem:special_linear} by
choosing $\phi=\nnz{\cdot}$.
For the second inequality, as the problem $\inf_{\G_i\in\GM_i} \nnz{\prod_i \T
\G_i}$ induces a tensor decomposition into $\nnz{\prod_i \T \hat{\G}_i}$
number of rank-1 tensors, by definition of the tensor rank we have the following inequality:
\[
\inf_{\G_i\in\GM_i} \nnz{\prod_i \T \G_i} \ge \rk(\T).
\]
\end{proof}

When $\phi=\nnz{\cdot}$, Lemma~\ref{lem:lifting} provides a link between
the rank of the tensor $\T$ and the ranks of the matrices or the vectors unfolded from $\T$.
For example, when $\GM=\UG$ and $\T\neq\0$, we have
\[
\inf_{\G\in\UG} \nnz{\G \vect (\T)} = 1.
\]
However, for the same $\T$ unfolded to a matrix $\fold^{-1}_I (\T)$, we have
\[
\inf_{\G_1\in\UG_1, \G_2\in\UG_2} \nnz{\G_1 \fold^{-1}_I(\T) \G_2} =
\rank(\fold^{-1}_I(\T)) \ge 1 = \inf_{\G\in\UG} \nnz{\G \vect (\T)}.
\]
Also, for the strong orthogonal rank of $\T$, we have
\[
\sorank\T = \inf_{\G_i\in\UG_i} \nnz{\prod_i \T\G_i} \ge
\rank(\fold^{-1}_I(\T)).
\]

Hence, intuitively, for a higher order tensor $\T$, we can only hope to find
decomposition with progressively ``denser" core than the matrices and the vectors
unfolded from $\T$. This can be describe more formally in the following lemma.

\begin{lemma}[Optimal core when unfoldable to optimal diagonal]
\label{lem:optimal-core}
If a tensor $\T$ admits a decomposition $\T=\prod_i \ZM \G_i$, where $\G_i\in
\GM_i$,
and there exists an index set grouping $J$ such that $\prodod_i\GM_i$ is
a sugroup of $\prodod_j\tilde{\GM}_j$, and
\[
\inf_{\prodod_j\tilde{\G}_j\in\prodod_j\tilde{\GM}_j}\varphi((\prodod_j
\tilde{\G}_j)\vect(\fold^{-1}_J(\ZM))) = \varphi(\vect(\fold^{-1}_J(\ZM))),
\]
then $\Z$ is the optimal sparse core in the following sense:
\[
\varphi(\vect(\ZM)) = \inf_{\G_i\in\GM_i}\varphi(\vect(\prod_i \T \G_i)) \; .
\]
\end{lemma}

\begin{proof}
We have
\begin{align*}
\varphi(\vect(\ZM)) \ge \inf_{\G_i\in\GM_i} \varphi((\prodod_i
\G_i)\vect(\ZM)) &\ge \inf_{\tilde{\G}_j\in\tilde{\GM}_j} \phi((\prodod_j
\tilde{\G}_j)\vect(\fold^{-1}_J(\ZM)))\\ &= \varphi(\vect(\fold^{-1}_J(\ZM))) =
\varphi(\vect(\ZM)).
\end{align*}
Hence,
\begin{align*}
\varphi(\vect(\ZM)) &=\inf_{\G_i\in\GM_i} \varphi((\prodod_i
\G_i)\vect(\ZM)) \\ & = \inf_{\G_i\in\GM_i} \varphi(\vect(\prod_i \ZM \G_i)) =
\inf_{\G_i\in\GM_i} \varphi(\vect(\prod_i \T \G_i)).
\end{align*}

\end{proof}

Applying Lemmas~\ref{lem:svd-cost},
and~\ref{lem:matrix-equivalence-cost} to Lemma~\ref{lem:optimal-core}, we immediately have the following theorem.

\begin{theorem}
\label{thm:unfolded-diagonal}
If a tensor $\T$ admits a decomposition $\T=\prod_i \ZM \G_i$, and there exists
an index set grouping $J$ such that $\fold^{-1}_J (\ZM)$ is of optimal shape
\wrt $f$ and $\GM$ in Table~\ref{tab:optimal-core}, then $\ZM$ is optimal
in the sense that
\[
\phi(\ZM) = \inf_{\G_i\in\GM_i} \phi(\prod_i \T\G_i).
\]
\end{theorem}

\begin{table}[!ht] \centering \small
\caption{Tensor decomposition generalized from matrix decomposition in the
GOO framework}
\begin{center}
\begin{tabular}{p{2.2cm} p{2.8cm} p{1.8cm} p{4.2cm}}
	\toprule
     Decomposition & Optimal core shape & Unit Group & Objective function
     \\
	\midrule Tensor SVD & unfoldable to some pseudo-diagonal matrix & $\UG$
    & $\sum_{ij} f(|x_{ij}|)$ where $f(\sqrt{x})$ is strictly
    concave, $f(0)\ge 0$\\
	\midrule Tensor Equivalence & unfoldable to $\lambda \I_n$ & $\SLG$
    & $\sum_{ij} f(|x_{ij}|)$ where $f(\sqrt{x})$ is
    strictly concave and increasing; $f(\sqrt{e^x})$ is convex; $f(0)\ge 0$\\
\hline
  \end{tabular}
\end{center} \label{tab:optimal-core}
\end{table}

If a tensor $\T$ admits a decomposition $\T=\prod_i \ZM \G_i$, where $\ZM$ is
superdiagonal, i.e., $z_{i_1, i_2, \cdots, i_n} \neq 0 \Rightarrow i_1 = i_2
=\cdots=i_n$, then by Theorem~\ref{thm:unfolded-diagonal}, $\ZM$ is optimal
under Tensor SVD as $\fold^{-1}_J (\ZM)$ is diagonal.
However, Theorem~\ref{thm:unfolded-diagonal} covers more cases than the
superdiagonal case, like the example below:
For example, consider the following \nth{4} order
tensor.

\begin{example}[Non-superdiagonalizable optimal tensor]
\label{example:non-superdiagonalizable} We consider \[\T\in\FB^{2\times 2\times
2\times 2},\; \vect{(\T)}=[{{{{1,0}, {0,1}}, {{0,0}, {0,0}}}, {{{0,0},
{0,0}}, {{1,0}, {0,1}}}}]^\top.\] We can unfold $\T$ to $\I_4$ with index set grouping $\{\{1,2\},\{3,4\}\}$.
Hence, $\T$ cannot be further ``sparsified" by GOO \wrt any matrix group,
even though it is not in superdiagonal form.
\end{example}

\begin{corollary}
There exist tensors that do not have a superdiagonal core under
any Tucker decomposition induced by GOO.
\end{corollary}

\begin{proof}
The tensor $\T$ in Example~\ref{example:non-superdiagonalizable} can be unfolded to a
scaled identity matrix. Hence, by Corollary~\ref{cor:special-linear-best} we have
\[
\inf_{\det(\prodod_i\G_i)=1} \nnz{\prod_i \T \G_i}
\ge \inf_{\G_i \in \SLG(n_i)} \nnz{\prod_i \T \G_i} = 4.
\]
This means that a Tucker decomposition of $\T$ induced by a GOO will have at least
four non-zero elements in the core matrix. However, the superdiagonal core can
only have at most two non-zero elements. Hence,  $\T$ does not have a superdiagonal core under any
Tucker decomposition induced by GOO.
\end{proof}

It is known that the minimal rank tensor decomposition in the Tucker model is not
unique. For example, for tensor $\T\in\FB^{2\times 2\times 2}$, we have
\begin{align}
\vect(\T)=[{{{1, 0}, {0, 3}},{{0, 0}, {0, 2}}}]^\top \textrm{,}
\end{align}
and there exist unitary matrices $\U_i\in\UG$ so that
\begin{align}
\vect(\prod_i \T \U_i)\approx
[{{{3.6055, 0},{0, 0.8320}},{{0, 0},{0, -0.5547}}}]^\top \textrm{.}
\end{align}

This means that GOO by $\UG$ \wrt different $f(x)=\|x\|_p$ may lead to
different optimum values. Hence, the class of entry-wise cost functions $f(x)=\|x\|_p,
0\le p<2$,  may not be used to induce sparsity  when the optimal core tensor cannot be
unfolded to a diagonal matrix. In other words, in the matrix case, any $p$,
$\|x\|_p, 0\le p<2$ can be used to find the sparest core; however, for a tensor
with order larger than 3, only $f(x)=\nnz{x}$ can be used for finding the
``sparsest" core of $\T$ under GOO.
In practice, this may be done by the following asymptotic formulation:
\[
\inf_{\G_i\in\UG} \nnz{\prod_i\T\G_i} = \lim_{p\to 0} \inf_{\G_i\in\UG}
\|\prod_i\T\G_i\|_p.
\]

In Section~\ref{sec:exps} we will present several concrete examples.

\section{Data Normalization}
\label{sec:normal}

Data normalization seeks to eliminate some arbitrary degrees of freedom in data.
For example, when we are concerned with shape of an object, its attitude and
position in space will become irrelevant. Given a point clouds, which is a
sequence of coordinates of points, we demonstrate a method to
obtain a representation of point clouds that does not depend on its attitude and
position in this section. We craft the method as a special case of GOO
with some particular choice of group and cost function.

\subsection{Shape Analysis: Matching vs. Normalization}

Point cloud data arise when interest points are extracted from images. If there
are $k$ points, each with a $d$-tuple coordinate, a $k\times d$ matrix can be
formed to describe the object.

Given point cloud data describing an object, shape matching tries to find an
object of the closest shape within a candidate set of shapes under some measure.
The shape space method for shape matching works by matching two objects with
known point-to-point correspondence over given group orbits.
For example, if an object described by $\A$ is known to be a rotated version of
another known object $\B$, we can find out parameters describing the rotation by
the following optimization formulation:
\[ g(\A, \B) \eqdef \inf_{\R\in\UG(d)}\|\A\R - \B \| \textrm{.} \] If there
are $n$ candidates $\{\B_i\}_{i=1}^n$, then the best matching object can be
found by $\arg\min_i g(\A, \B_i)$.

However, to make the above method  work, a point correspondence procedure must
be established in the first place, which means that the same row of $\A$ and
$\B$ should refer to the same point. This meets difficulties in real world data
applications because
\begin{enumerate}
  \item[(1)] $\A$ and $\B$ may have different numbers of rows;
  \item[(2)] $\A$ and $\B$ may have many rows, leading to exponential number of
  possible correspondence.
\end{enumerate}

Here we present a method to match point cloud by using normalization to simplify
matching. The first step of the method is normalizing each of objects
$\{\A_i\}_{i=1}^m$ and $\{\B_i\}_{i=1}^n$ with following optimization:
\[\hat{\M} = \arginf_{\R\in\UG(d)} \phi(\M \R) \textrm{.} \] As in
Section~\ref{subsec:induced-matrix-decomp}, the above optimization leads to
following decompositions:
\[ \A_i = \hat{\A}_i \R_i\text{ and }\B_i = \hat{\B}_i \U_i\text{,} \] where
$\R_i,\U_i\in\GM_i$ for some group $\GM_i$.

The second step of the method carries out matching of objects $\{\A_i\}_{i=1}^m$
against $\{\B_i\}_{i=1}^n$ by using the normalized forms
$\{\hat{\A}_i\}_{i=1}^m$ and $\{\hat{\B}_i\}_{i=1}^n$. Matching between
$\{\hat{\A}_i\}_{i=1}^n$ and $\{\hat{\B}_i\}_{i=1}^n$ is expected to be simpler
because less degrees of freedom remain after normalization.

A well-known data normalization method is Principal Component Analysis (PCA),
which eliminates the following degrees of freedom: translation, scaling and
rotation.
As any rigid body movement can be expressed as combination of translation and
rotation, PCA provides a method to standardize data \wrt the rigid body
movement.
However, there may be other distortions of data. Thus, we discuss using general
group for normalizing point cloud data to eliminate the effect of non-rigid body
transforms. An illustrative example has been shown in
Figure~\ref{fig:special_linear} and Figure~\ref{fig:special_linear_3d} in
Section~1, where we see that normalized point clouds can be
matched by enumerating a small number of orientation.

\subsection{Normalization of Point Cloud Data by the Special Linear Group}

In this section we use group $\BP^\top\otimes\I_m, \BP\in\SLG(n)$ for
normalization of point cloud data.

Here we assume the distortions to the point
cloud data are of a few categories of degrees of freedom: including
mirroring, rotation, shearing and squeezing, which we seek to eliminate using
the special linear group orbit.

\begin{lemma} The
special linear group can represent any combination of mirroring, rotation,
shearing and squeezing operations for point cloud data.
\end{lemma}
\begin{proof}
Every special linear matrix $\M$
can be QR decomposed as $\M = \Q \R$, and $\R$ can be decomposed into $\R = \D
\U$ where $\D$ is diagonal and $\U\in\UUTG$. Accordingly, we have a decomposition $\M =
\Q \D \U$,
where $\U$ models the shearing operation and $\Q$ models the
rotation. As $\det (\M) = 1$, $|\det \D| =
|\frac{\det \M}{\det \Q \det \U}| = 1$.
Hence, the diagonal matrix $\D$ is the squeezing operation (optionally with the
mirror operation). Hence the action of $\G\in\SLG$ applied to $\M$ is equivalent
to the sequential application of rotation, squeezing, mirroring, and shearing. Now as
the special linear group is a group, arbitrary composition of these operations
can still be represented as some $\G'\in\SLG$.
\end{proof}

We show that for some point clouds, the normalized form are exactly the
axis-aligned hypercubes.
\begin{lemma}
Given a matrix $\M\in\RB^{n\times d}$, if $\Poly{(\M)}$ is an axis-aligned hypercube and
 $\det{\G}=1$, then
\[
\|\M\G\|_\infty \ge \|\M\|_\infty.
\]
\end{lemma}
\begin{proof}
First note that as $\det{(\G)}=1$, given a Lebesgue measure $\mu$, we
have:
\[
\mu(\Poly{(\M\G)}) = \mu(\Poly{(\M)})\textrm{.}
\]
We can construct a bounding box $\Poly{(\Z)}$ for $\Poly{(\M\G)}$ with center at the
origin and edge length $2\|\M\G\|_\infty$. Note that $\Poly{(\Z)}$ is also
a hypercube and $\|\Z\|_\infty=\|\M\G\|_\infty$. We thus have
\[
\mu(\Poly{(\Z)}) \ge \mu(\Poly{(\M\G)})=\mu(\Poly{(\M)})\textrm{.}
\]
Because for any axis-aligned hypercube $\Poly{(\Y)}$ we have
$\mu(\Poly{\Y})=2^d\|\Y\|_\infty^d$, the following holds:
\[
2^d\|\M\G\|_\infty^d=2^d\|\Z\|_\infty^d \ge 2^d\|\M\|_\infty^d \textrm{.}
\]
\end{proof}

By Lemma~\ref{lem:criteria-infimum}, we can prove the following corollary.
\begin{corollary}\label{thm:special-linear-normalize}
Let $\M\in\RB^{n\times k}$ be a matrix such that $\Poly(\M)$ can be transformed
by $\G\in\SLG(k)$ into a hypercube.
Then the following optimization problem attains its optimum when $\Poly(\M \hat{\G})$
is a hypercube:
\[
\inf_{\G\in\SLG(n)} \|\M \G\|_\infty
\]
\end{corollary}

\begin{remark}
The optimization problem in Theorem~\ref{thm:special-linear-normalize} is not convex.
For example, in $\RB^2$, a square can be rotated by 90 degrees, 180 degrees,
and 270 degrees while still being a square. Nevertheless, the degree of freedom associated with rotation,
squeezing and shearing described by the special linear group is reduced to
only one of four configurations. The three other optimal $\hat{\G}$ can be
enumerated when one optimal $\G$ is known.
\end{remark}

\begin{remark}
Optimality of $\inf_{\G\in\SLG(2)} \|\M \G\|_\infty$ depends
on whether $\Poly(\M \G)$ is a parallelogram. In practice, we
find the above method works well in normalizing general point data, especially
for those arise in shape recognition. In
Section~\ref{sec:exps} we will present several concrete examples.
\end{remark}

\section{Numerical Algorithm and Examples}
\label{sec:exps}
\subsection{Algorithm}\label{subsec:algorithm}
Most of optimizations involved in this paper are constrained optimization
problem of the following form:
\begin{align}\label{opt:constrained-group}
\inf_{\G\in\GM} \; \varphi(\G \vect(\M)) \text{,}\end{align}
where $\GM$ is a unit group. When $\GM$ is a Lie group, alternatively we
can turn the above optimization to another constrained optimization:
\[ \inf_{\G \in \gm} \; \varphi(\exp(\G) \vect(\M))\text{,} \] where $\exp(\G)$
is matrix exponential of matrix $\G$ and $\gm$ is the Lie algebra associated with
Lie group $\GM$. This formulation may have constraints that are easier to encode
in numeric software. For example, in \[ \inf_{\G_1\in\SOG(m), \G_2\in\SOG(n)}
\varphi(\vect(\G_2 \M \G_1^\top)) \text{,}\] we have
$\GM=\SOG(m)\otimes\SOG(n)$, $\gm = \{\Z_2\oplus\Z_1:
\Z_1\in\FB^{n\times n}, \Z_1 + \Z_1^\top=0, \Z_2\in\FB^{m\times m},
\Z_2 + \Z_2^\top=0\}$.
Hence, we can turn the optimization into a
constrained optimization over $\gm$ as
\[\inf_{\Z\in\gm}\varphi(\exp(\Z)\vect(\M))\text{.}\]
Moreover, in this particular case, we can turn the above optimization into an
unconstrained optimization:\[ \inf_{\G_1\in\FB^{n\times n}, \G_2\in\FB^{m\times m}}
\varphi(\vect(\exp(\G_2-\G_2^\top) \M\exp(\G_1-\G_1^\top)^\top)) \text{.}\]

In Table~\ref{tab:lie-group-encoding} we list a few more cases when the
constrained optimization of Formula~\ref{opt:constrained-group} can be turned
into an unconstrained optimization.

\begin{table}[!ht] \centering \small
\caption{Encoding of constraints for Lie groups}
\begin{center}
\begin{tabular}{p{2cm} p{4cm} p{4cm}}
    \toprule
     Lie group & Lie algebra & Encoding of Constraint \\
    \midrule $\SOG$ & $\{\Z: \Z+\Z^\top=0\}$ & $\Z = \X - \X^\top$\\
    \midrule $\LUTG$ & $\{\Z: \dg\Z=\0, \Z\in\LUTG\}$ & $\Z = \X \odot
    [{\mathbb I}_{i > j}]$\\
    \midrule $\UUTG$ & $\{\Z: \dg\Z=\0, \Z\in\UUTG\}$ & $\Z = \X \odot
    [{\mathbb I}_{i < j}]$\\
    \midrule $\SLG$ & $\{\Z: \tr\Z=0\}$ & $\Z = \X - (\tr \X) \v \v^\top$,
    where $\v=[1, 0, 0, \ldots, 0]^\top$\\
    \midrule $\GM_1\otimes \GM_2$ & $\gm_1\oplus\gm_2$ & \\
\hline
  \end{tabular}
\end{center} \label{tab:lie-group-encoding}
\end{table}

The exponential mapping used for
optimization over Lie groups is related to other optimization on manifold
methods \cite{udriste1994convex} \cite{edelman1998geometry} \cite{absil2009optimization}.

In this section all numerical optimizations are solved by Nelder-Meld heuristic
global optimization algorithm \cite{nelder1965simplex} implemented in Mathematica\texttrademark\
9.0.0, unless noted.

\subsection{GOO Inducing Matrix Decomposition}

We empirically illustrate several examples of GOO inducing matrix decomposition.
Due to the large amount of computation required by Nelder-Meld algorithm, here
we only give a few examples involving small matrices.

\begin{example}[Compute SVD of a $3\times 3$ real matrix] \label{exp:svd}
Given a matrix $\M$:
\[
\M\approx\left[
\begin{array}{ccc}
 0.17658 & 0.517888 & 0.448587 \\
 0.214066 & 0.718154 & 0.849892 \\
 0.796042 & 0.197801 & 0.233489
\end{array}
\right],
\]
the SVD of $\M = \U \Lam \V^\top$ is given as
\[
\U \approx
\left[
\begin{array}{ccc}
 -0.483076 & -0.175226 & -0.857865 \\
 -0.768129 & -0.385453 & 0.511276 \\
 -0.420256 & 0.905937 & 0.0516068
\end{array}
\right],
\]
\[
\Lam\approx \left[
\begin{array}{ccc}
 1.43557 & 0. & 0. \\
 0. & 0.66535 & 0. \\
 0. & 0. & 0.0910448
\end{array}
\right],
\]
\[
\V\approx \left[
\begin{array}{ccc}
 -0.406999 & 0.913369 & -0.0104708 \\
 -0.616441 & -0.28311 & -0.734745 \\
 -0.674057 & -0.292586 & 0.678263
\end{array}
\right].
\]

We now use the  optimization problem:
\[\inf_{\U, \V\in\SOG(3)} \|\U^\top\M\V\|_1\]
to find the
SVD of $\M$. A numerical solution, produced by the heuristic global
optimization, is given as
\[
\hat{\U}\approx\left[
\begin{array}{ccc}
 0.483076 & -0.857865 & -0.175227 \\
 0.768129 & 0.511277 & -0.385453 \\
 0.420256 & 0.0516062 & 0.905937
\end{array}
\right],
\]
\[
\hat{\U}^\top\M\hat{\V} \approx \left[
\begin{array}{ccc}
 9.16351\times 10^{-10} & -8.52198\times 10^{-9} & {\bf 1.43557} \\
 -4.58791\times 10^{-7} & {\bf -9.10448\times 10^{-2}} & 4.36886\times 10^{-9} \\
{\bf 6.6535\times 10^{-1}} & 8.13893\times 10^{-8} & 2.67077\times 10^{-10} \\
\end{array}
\right],
\]
\[
\hat{\V}\approx\left[
\begin{array}{ccc}
 0.913369 & 0.010471 & 0.406999 \\
 -0.28311 & 0.734745 & 0.616441 \\
 -0.292585 & -0.678263 & 0.674057 \\
\end{array}
\right].
\]

Note that $\hat{\U}$, $\hat{\V}$, $\hat{\U}^\top\M\hat{\V}$ are permuted
approximations of $\U$, $\V$, $\Lam$ modulo sign, respectively.

\end{example}
\begin{example}[Compute QR of a $3\times 3$ matrix] We use the same $\M$ as in
Example~\ref{exp:svd}. QR decomposition of $\M$ is given by $\M = \Q\D\R$ where
\[\Q\approx
\left[
\begin{array}{ccc}
 -0.20946 & -0.541716 & -0.814046 \\
 -0.253927 & -0.773817 & 0.580283 \\
 -0.944271 & 0.328254 & 0.0245274 \\
\end{array}
\right],
\]
\[\D\approx\left[
\begin{array}{ccc}
 -0.843023 & 0. & 0. \\
 0. & -0.771339 & 0. \\
 0. & 0. & 0.133735 \\
\end{array}
\right],\]
\[\R\approx\left[
\begin{array}{ccc}
 1. & 0.566548 & 0.628984 \\
 0. & 1. & 1.0683 \\
 0. & 0. & 1. \\
\end{array}
\right].\]

We use  optimization of the form
 \[\inf_{\Q\in\SOG(3),\R\in\UUTG(3)}
\|\Q^\top\M\R^{-1}\|_1\]
to find QR of $\M$. A numerical solution, produced
by the heuristic global optimization, is given  as
\[\hat{\Q}\approx
\left[
\begin{array}{ccc}
 0.814046 & 0.541716 & 0.20946 \\
 -0.580283 & 0.773817 & 0.253927 \\
 -0.0245274 & -0.328254 & 0.944271 \\
\end{array}
\right],
\]
\[\hat{\Q}^\top\M\hat{\R}^{-1}\approx
\left[
\begin{array}{ccc}
 4.65594\times 10^{-9} & -7.03107\times 10^{-9} & \bf -1.33735\times 10^{-1} \\
 2.32509\times 10^{-8} & \bf 7.71339\times 10^{-1} & 6.05164\times 10^{-9} \\
 \bf 8.43023\times 10^{-1} & 5.62326\times 10^{-3} & 3.42317\times 10^{-9} \\
\end{array}
\right],
\]
\[\hat{\R}\approx
\left[
\begin{array}{ccc}
 1 & 0.559878 & 0.621858 \\
 0 & 1 & 1.0683 \\
 0 & 0 & 1 \\
\end{array}
\right].
\]
Note that $\hat{\Q}$, $\hat{\Q}^\top\M\hat{\R}^{-1}$, and $\hat{\R}$ are
permuted approximations of $\Q$, $\D$, $\R$ modulo sign,
respectively.
Note although there is a significant difference between $\R$ and $\hat{\R}$ as
$\|\R-\hat{\R}\|_F\approx0.00976074$, the decomposition is still good
approximation as we have \[\|\hat{\Q}\hat{\D}\hat{\R} -
\M\|_F\approx3.01503\times 10^{-16}\text{.}\]
\end{example}

\begin{example}[Compute matrix equivalence decomposition of a $3\times 3$
matrix] We use the same $\M$ as in Example~\ref{exp:svd}. Matrix equivalence
decomposition of $\M$ is not unique. Anyway the optimal core modulo sign
and permutation would be \[\D\approx
\left[
\begin{array}{ccc}
 0.44304 & 0. & 0. \\
 0. & 0.44304 & 0. \\
 0. & 0. & 0.44304
\end{array}
\right].
\]

We use the optimization \[\inf_{\A\in\SLG(3),\B\in\SLG(3)}
\|\A^{-1}\M\B^{-1}\|_1\] to find the matrix equivalence decomposition of $\M$. A
numerical solution produced by the heuristic global optimization is given as
\[\hat{\A}^{-1}\M \hat{\B}^{-1}\approx
\left[
\begin{array}{ccc}
 2.00925\times 10^{-9} & {\bf 4.42891\times 10^{-1}} & 2.77141\times 10^{-9} \\
 1.18223\times 10^{-9} & 2.10486\times 10^{-8} & {\bf 4.43117\times 10^{-1}} \\
 {\bf 4.43112\times 10^{-1}} & -7.24341\times 10^{-9} & 1.79685\times 10^{-10} \\
\end{array}
\right].
\]

\end{example}

\begin{example}[Compute LU of a $3\times 3$ matrix] We use the same $\M$ as in
Example~\ref{exp:svd}. The LU decomposition of $\M$ without pivoting is given by
\[\L\approx
\left[
\begin{array}{ccc}
 1 & 0 & 0 \\
 1.21229 & 1 & 0 \\
 4.50812 & -2.36585\times 10^1 & 1 \\
\end{array}
\right],
\]
\[\U\approx
\left[
\begin{array}{ccc}
 1.7658\times 10^{-1} & 5.17888\times 10^{-1} & 4.48587\times 10^{-1} \\
 2.77556\times 10^{-17} & 9.03226\times 10^{-2} & 3.06073\times 10^{-1} \\
 1.11022\times 10^{-16} & 0. & 5.45245 \\
\end{array}
\right]
\textrm{.}\]

We use the optimization \[\inf_{\L\in\LUTG(3)}
\|\Delta \odot (\L^{-1}\M)\|_1\] to find the LU of $\M$, where
$\Delta_{ij}={\mathbb I}_{i>j}$. A numerical solution produced
by the heuristic global optimization is given as
\[\hat{\L}\approx
\left[
\begin{array}{ccc}
 1 & 0 & 0 \\
 1.21229 & 1 & 0 \\
 4.50812 & -2.36585\times 10^1 & 1 \\
\end{array}
\right]
\textrm{,}\]
\[\hat{\U}\approx
\left[
\begin{array}{ccc}
 1.7658\times 10^{-1} & 5.17888\times 10^{-1} & 4.48587\times 10^{-1} \\
 -1.64336\times 10^{-17} & 9.03226\times 10^{-2} & 3.06073\times 10^{-1} \\
 2.22045\times 10^{-16} & 4.44089\times 10^{-16} & 5.45245 \\
\end{array}
\right]
\textrm{.}\]

\end{example}

\begin{example}[Compute Cholesky decomposition of a $3\times 3$ matrix] We use
the $\M^\top\M$ as input with $\M$ from Example~\ref{exp:svd}. The Cholesky
decomposition of $\M^\top\M$ is given by $\M^\top\M = \U^*\U$ where
\[\U^*\approx
\left[
\begin{array}{ccc}
 0.843023 & 0. & 0. \\
 0.477613 & 0.771339 & 0. \\
 0.530248 & 0.824024 & 0.133735 \\
\end{array}
\right]
\textrm{.}
\]

We use the optimization \[\inf_{\L\in\LUTG(3)}
\|\Delta \odot (\L^{-1}\M)\|_1\] to find the LU of $\M$, where
$\Delta_{ij}={\mathbb I}_{i>j}$. A numerical solution
is given as
\[\hat{\L}\boldsymbol\Lambda\approx
\left[
\begin{array}{ccc}
 0.843023 & 0. & 0. \\
 0.477613 & 0.771339 & 0. \\
 0.530248 & 0.824024 & 0.133735 \\
\end{array}
\right].
\]
Here $\boldsymbol\Lambda$ is a diagonal matrix with square root of
diagonals of $\hat{\L}^{-1}\M$ as its diagonals.

\end{example}

\begin{example}[Compute Schur decomposition of a $3\times 3$ matrix] We use the
same $\M$ as in Example~\ref{exp:svd}. The Schur decomposition of $\M = \Q\U\Q^{-1}$
is given by
\[\Re\Q\approx
\left[
\begin{array}{ccc}
 -4.4809\times 10^{-1} & 2.13496\times 10^{-2} & -3.11518\times 10^{-1} \\
 -6.81147\times 10^{-1} & -5.00762\times 10^{-1} & 1.85125\times 10^{-3} \\
 -4.25133\times 10^{-1} & 7.79816\times 10^{-1} & 3.25373\times 10^{-1} \\
\end{array}
\right]
\textrm{,}\]
\[\Im\Q\approx
\left[
\begin{array}{ccc}
 -1.91558\times 10^{-1} & 2.76329\times 10^{-1} & -7.67245\times 10^{-1} \\
 -2.91189\times 10^{-1} & -2.4227\times 10^{-2} & 4.47096\times 10^{-1} \\
 -1.81744\times 10^{-1} & -2.52434\times 10^{-1} & 9.23392\times 10^{-2} \\
\end{array}
\right]
\textrm{,}\]
\[\Re\U\approx
\left[
\begin{array}{ccc}
 1.38943 & -2.5768\times 10^{-1} & -1.15903\times 10^{-1} \\
 0. & -1.30605\times 10^{-1} & -7.89372\times 10^{-2} \\
 0. & 0. & -1.30605\times 10^{-1}
\end{array}
\right]
\textrm{,}\]
\[\Im\U\approx
\left[
\begin{array}{ccc}
 -1.38778\times 10^{-16} & 2.11604\times 10^{-1} & 3.88987\times 10^{-2} \\
 0. & 2.13379\times 10^{-1} & -5.69018\times 10^{-1} \\
 0. & 0. & -2.13379\times 10^{-1} \\
\end{array}
\right]
\textrm{,}\]

We can use the optimization \[\inf_{\Q\in\UG(3)}
\|\Delta\odot(\Q^{-1}\M\Q)\|_1\] to find the Schur decomposition of $\M$, where
$\Delta_{ij}={\mathbb I}_{i>j}$. A numerical solution  is given as
\[\Re\hat{\Q}\approx
\left[
\begin{array}{ccc}
 -0.159122 & 0.456745 & -0.924103 \\
 -0.697698 & 0.596808 & 0.46622 \\
 0.777452 & 0.665151 & 0.227028
\end{array}
\right]
\textrm{,}\]
\[\Im\hat{\Q}\approx
\left[
\begin{array}{ccc}
 -0.288006 & -0.0686255 & 0.0156732 \\
 0.161731 & 0.0500731 & 0.177932 \\
 0.0861939 & 0.00219548 & -0.301602
\end{array}
\right]
\textrm{,}\]
\[\Re\hat{\U}\approx
\left[
\begin{array}{ccc}
 \bf -1.30605\times 10^{-1} & -3.71633\times 10^{-1} & -7.07291\times 10^{-1} \\
 -2.09852\times 10^{-9} & \bf 1.38943 & -1.09975\times 10^{-1} \\
 -1.39859\times 10^{-8} & 9.97677\times 10^{-9} & \bf -1.30604\times 10^{-1}
\end{array}
\right]
\textrm{,}\]
\[\Im\hat{\U}\approx
\left[
\begin{array}{ccc}
\bf -2.13379\times 10^{-1} & -1.95892\times 10^{-2} & 2.66849\times 10^{-2} \\
 2.82526\times 10^{-9} & \bf 2.10889\times 10^{-9} & -1.05424\times 10^{-1} \\
 9.09511\times 10^{-9} & 8.28696\times 10^{-9} & \bf 2.13379\times 10^{-1} \\
\end{array}
\right]
\textrm{.}\]
We note that $\hat{\U}$ is permuted approximation of $\U$ modulo sign.

\end{example}
\subsection{GOO Inducing Tensor Decomposition}

We empirically illustrate several examples of GOO inducing tensor decomposition.
Due to the large amount of computation required by the Nelder-Meld algorithm, here
we only give  examples involving small-size tensors.

\begin{example}[Non-uniqueness of strong-orthogonal decomposition]
\label{ex:non-unique-strong-orthogonal-decomp}
Consider tensor $\AM$ as in Example 3.3 of \cite{Kolda:2001:OTD:587708.587830},
which we reproduce below:
\begin{align}\label{eq:kolda-example-3.3}
\AM=\sigma_1 \a\otimes \b\otimes \b+\sigma_2 \b\otimes \b\otimes \b+\sigma_3
\a\otimes \a\otimes \a\text{,}
\end{align}
where $\sigma_1>\sigma_2>\sigma_3, \a\perp \b, \|\a\|=\|\b\|=1$.

We note that Formula~(\ref{eq:kolda-example-3.3}) is already a strong orthogonal
decomposition of $\AM$. Nevertheless, an alternative strong orthogonal
decomposition is given therein as
\begin{align}\label{eq:kolda-example-3.3-alt}
\AM=\hat{\sigma}_1
\hat{\a}\otimes\b\otimes\b+\hat{\sigma}_2\hat{\a}\otimes\a\otimes\a+\hat{\sigma}_3\hat{\b}\otimes\a\otimes\a
\text{,}\end{align}
where
\[
\hat{\sigma}_1 = \sqrt{\sigma_1^2+\sigma_2^2},\qquad
\hat{\sigma}_2=\frac{\sigma_1\sigma_3}{\hat{\sigma}_2},\qquad\hat{\sigma}_3=\frac{\sigma_2\sigma_3}{\hat{\sigma}_1}\text{,}
\]
\[
\hat{\a} = \frac{\sigma_1 \a + \sigma_2 \b}{\hat{\sigma}_1},\text{ and }\;
\hat{\b} = \frac{\sigma_2 \a - \sigma_1 \b}{\hat{\sigma}_1}\text{.}
\]

Without loss of generality, we let $\sigma_1 =3, \sigma_2 = 2, \sigma_3 = 1,
\a=[1,0]^\top$, and $\b=[0,1]^\top$. Then
\[
\vect{(\AM)} = [1,0,0,0,0,0,3,2]^\top\text{,}
\]
\[
\hat{\sigma}_1\approx3.60555,\qquad\hat{\sigma}_2\approx0.832050,\qquad\hat{\sigma}_3\approx
0.5547002.
\]
In framework of GOO, we can induce a strong orthogonal decomposition of tensor
$\AM$ by the following optimization:
\[
\inf_{\G_1\in \SOG(2), \G_2\in \SOG(2), \G_3\in \SOG(2)} \|\AM \times_1
\G_1\times_2 \G_2\times_3 \G_3\|_1 \text{.}
\]

One numerical solution of core tensor is:
\[\vect(\AM\times_1\hat{\G}_1\times_2\hat{\G}_2\times_3\hat{\G}_3) \approx
\left[
\begin{array}{c}
 \bf 3.60555 \\
 1.24142\times 10^{-8} \\
 1.92366\times 10^{-10} \\
 -2.91798\times 10^{-9} \\
 -1.80012\times 10^{-8} \\
 -6.43786\times 10^{-10} \\
 \bf 8.3205\times 10^{-1} \\
 \bf -5.547\times 10^{-1} \\
\end{array}
\right]
\text{.}
\]

Note that the large nonzero values (in bold) are approximations of $\hat{\sigma}_1$,
$\hat{\sigma}_2$, and $\hat{\sigma}_3$, modulo sign.
\end{example}

\begin{example}[The Special Linear Group finds Sparser Core in Tensor
Decomposition] $\AM$ is  given as in
Example~\ref{ex:non-unique-strong-orthogonal-decomp}. We
can induce a ``sparser'' decomposition of tensor with the following GOO:
\[
\inf_{\G_1\in \SLG(2), \G_2\in \SLG(2), \G_3\in \SLG(2)} \|\AM \times_1
\G_1\times_2 \G_2\times_3 \G_3\|_1 \text{.}
\]
One numerical solution of core tensor is:
\[\vect(\AM\times_1\hat{\G}_1\times_2\hat{\G}_2\times_3\hat{\G}_3) \approx
\left[
\begin{array}{c}
 \bf 1.41421 \\
 -5.7745\times 10^{-9} \\
 -9.00468\times 10^{-9} \\
 1.49267\times 10^{-9} \\
 -4.79358\times 10^{-9} \\
 -8.18472\times 10^{-9} \\
 -2.80551\times 10^{-9} \\
 \bf 1.41421 \\
\end{array}
\right]
\text{.}
\]

Note that there are only two significant nonzero values (in bold), in contrast to
three in the strong orthogonal decomposition. Since
$\AM\times_1\hat{\G}_1\times_2\hat{\G}_2\times_3\hat{\G}_3$ is superdiagonal, it
is the ``sparsest'' core tensor under any Tucker decompositions.
\end{example}

\begin{example}[A tensor that does not have Superdiagonal Form but is also of
Lowest Rank under any Tucker
Decomposition]
We give a numerical solution to Example~\ref{example:non-superdiagonalizable}
where
\[\T\in\FB^{2\times 2\times
2\times 2},\; \vect{(\T)}=[{{{{1,0}, {0,1}}, {{0,0}, {0,0}}}, {{{0,0},
{0,0}}, {{1,0}, {0,1}}}}]^\top.\]

The solution to
\[
\inf_{\G_1\in \SLG(2), \G_2\in \SLG(2), \G_3\in \SLG(2),\G_4\in\SLG(2)} \|\AM
\times_1 \G_1\times_2 \G_2\times_3 \G_3\times_4\G_4\|_1
\]
is
\[\vect(\TM\times_1\hat{\G}_1\times_2\hat{\G}_2\times_3\hat{\G}_3\times_4\hat{\G}_4)
\approx \left[
\begin{array}{c}
 \bf 1. \\
 5.85196\times 10^{-9} \\
 2.40735\times 10^{-10} \\
 \bf -1. \\
 -2.39741\times 10^{-9} \\
 -1.40295\times 10^{-17} \\
 -5.77138\times 10^{-19} \\
 2.3974\times 10^{-9} \\
 4.74159\times 10^{-9} \\
 2.77475\times 10^{-17} \\
 1.14146\times 10^{-18} \\
 -4.74158\times 10^{-9} \\
 \bf 9.99999\times 10^{-1} \\
 5.85194\times 10^{-9} \\
 2.40734\times 10^{-10} \\
 \bf -9.99998\times 10^{-1} \\
\end{array}
\right]\text{.}
\]
Hence there are four significant nonzero values even under GOO \wrt the special
linear group.
\end{example}

\subsection{Normalization of point cloud \wrt special linear group}
Here we apply the GOO defined in Section~\ref{sec:normal} to a publicly
available set of 2D point cloud data
\href{http://vision.lems.brown.edu/content/available-software-and-databases/#Datasets-Shape}{here}.

As the optimization variable only consists of a small matrix
$\M\in\FB^{2\times 2}$, we are able to deal with large point clouds consisting
of more than thousands of points.

The detailed steps are  as follows:
\begin{algorithm}\label{alg:special-linear}
\begin{enumerate}
  \item[Step 1] Normalize the point cloud corresponding to $\M$ \wrt special linear
  group as \[\hat{\M} = \arginf_{\G\in\SLG(n)} \|\M \G\|_\infty\text{.}\]
  \item[Step 2] (Optional) Let $\hat{\M}_x$ and $\hat{\M}_y$ be two columns of
  $\hat{\M}$. We can use a simple criterion to select one from four possible forms of normalized
  point clouds: $[\hat{\M}_x, \hat{\M}_y]$, $[-\hat{\M}_y, \hat{\M}_x]$,
  $[-\hat{\M}_x, -\hat{\M}_y]$, and $[\hat{\M}_y, -\hat{\M}_x]$ to further
  eliminate ambiguity.
  An example  is to pick the matrix $\hat{\hat{\M}}$ that minimizes $\phi(\X) = \|g(\X)\|_F$
  where $g(x) = \max(0, x)$. $\hat{\hat{\M}}$ is called a canonical form of $\M$
  in this section.
\end{enumerate}
\end{algorithm}
\begin{remark}
Step~2 in Algorithm~\ref{alg:special-linear} is found to be useful in
eliminating the ambiguity in orientation in some circumstances. However, even if
Step~2 fails or is skipped, one can still use $\hat{\M}$ as ``canonical''
form and enumerate the few number of possible orientations. The result of
Algorithm~\ref{alg:special-linear} without Step~2 is shown in
Figure~\ref{fig:special_linear} and Figure~\ref{fig:special_linear_3d}.
\end{remark}

In Figure~\ref{fig:normal_form} we perform a side-by-side comparison of results
of several normalization techniques.
The point clouds in the row marked with ``Distorted'' are produced by applying
random shearing, mirroring, squeezing and rotation to the same point cloud. The point
clouds in the row marked with ``PCA'' are results of applying PCA to the
matrices corresponding to the distorted point clouds in the ``Distorted'' row. It
can be seen that PCA can remove the degree of freedom corresponding to rotation in the input data, but fails to
remove effect of squeezing and shearing. The row marked with ``GOO\_SO'' is
produced by using GOO with orthogonal group:
\[
\inf_{\G\in\OG(n)} \|\M \G\|_\infty\text{.}\]
We can see that effect of
rotation is removed but effects of squeezing and shearing remain. The row
marked with ``GOO\_SL'' is the canonical forms of matrices corresponding
to point clouds derived by Algorithm~\ref{alg:special-linear}. We can see that
the normalized point clouds are approximately the same, and effects of rotation, squeezing and shearing are
almost completely eliminated.

\begin{figure}
\subfigtopskip = 0pt
\begin{center}
\centering
\subfigure{\includegraphics[height=64mm,
width=130mm]{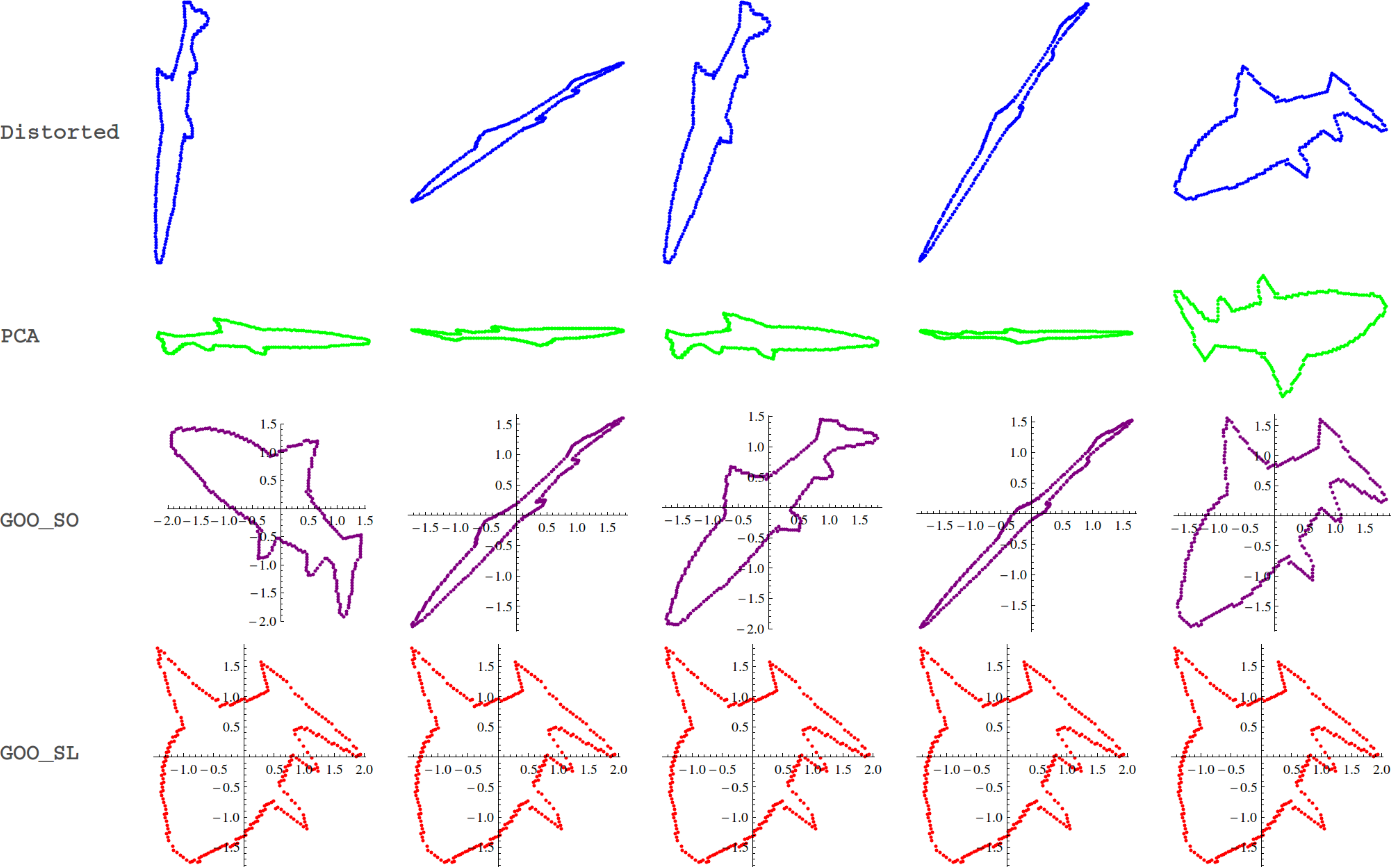}}
\subfigure{\includegraphics[height=64mm,
width=130mm]{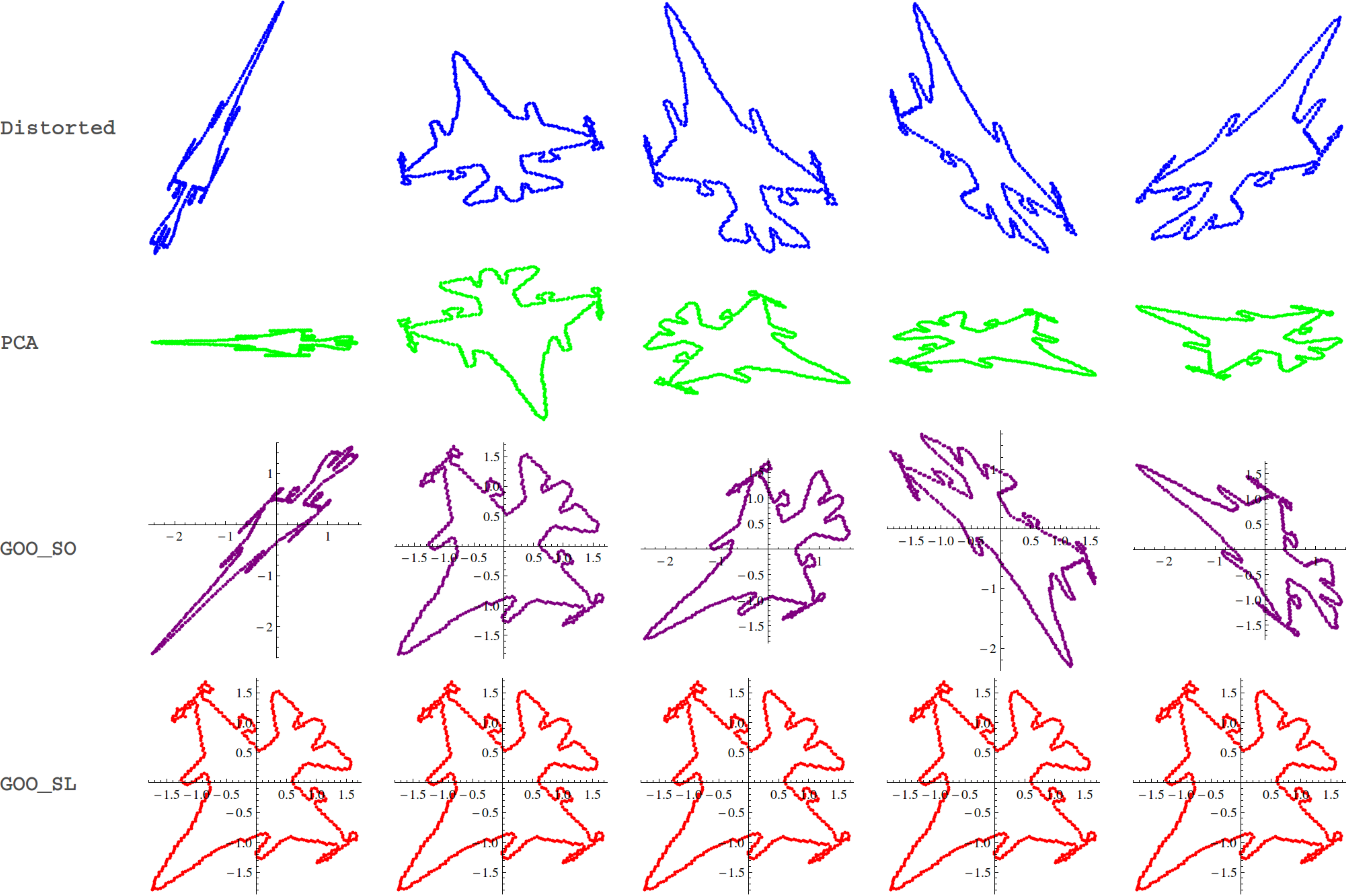}}
\subfigure{\includegraphics[height=64mm,
width=130mm]{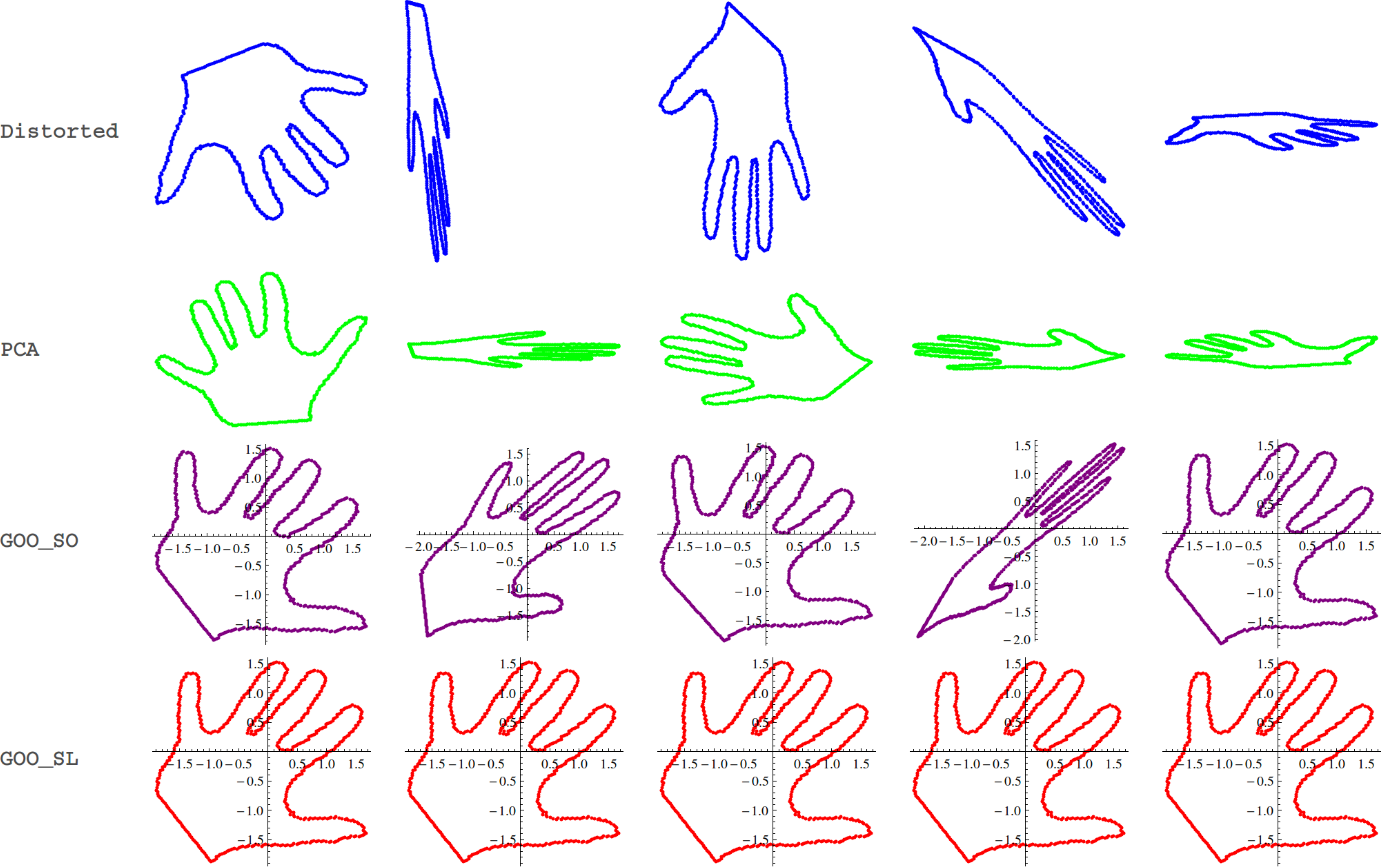}}

\end{center}
   \caption{This figures show the results of normalizing distorted point clouds
   by different methods.
   The rows marked with ``Distorted'' consist of distorted point clouds used as
   input to various normalization methods.
   The rows marked with ``PCA'' contain results of normalization by principal
   component analysis.  It can be seen that the effects of
   rotational distortion have been partially eliminated, but results of shearing
   and squeezing remain.
   The rows marked with ``GOO\_SO'' contain results of normalization using
   special orthogonal group in GOO. It can be seen that the effects of
   rotational distortion have been partially eliminated, but results of shearing
   and squeezing remain.
   The rows marked with ``GOO\_SL'' contain results of
   normalization using Algorithm~\ref{alg:special-linear}, where it can be seen that the algorithm
   can produce approximately the same point clouds after eliminating distortions
like shearing, squeezing and rotation.}
\label{fig:normal_form}
\end{figure}


In Figure~\ref{fig:sampling_rate} we study the impact of number of points on
the canonical form found by the GOO.
We can see that though the number of points in the canonical form vary between
180 and 260, the canonical form is nearly the same, module different
orientations. In this case although Step~2 in
Algorithm~\ref{alg:special-linear} cannot completely eliminate the ambiguity of
four possible orientations of point clouds, we can simply remove this ambiguity by
enumerating all four possible orientations when doing comparison.
This property means that when comparing shape of two point clouds, it is not necessary to
require two point clouds to have exactly the same number of points when we are
comparing based on the canonical forms.

\begin{figure}
\subfigtopskip = 0pt
\begin{center}
\centering
\subfigure{\includegraphics[height=48mm,
width=130mm]{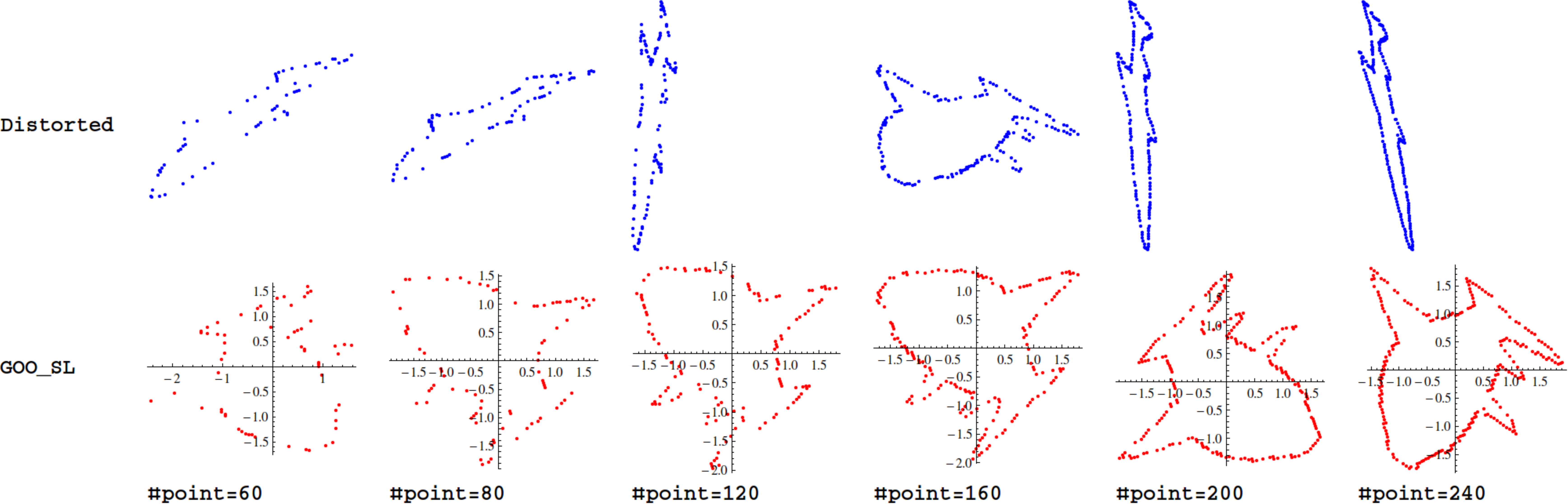}}
\subfigure{\includegraphics[height=48mm,
width=130mm]{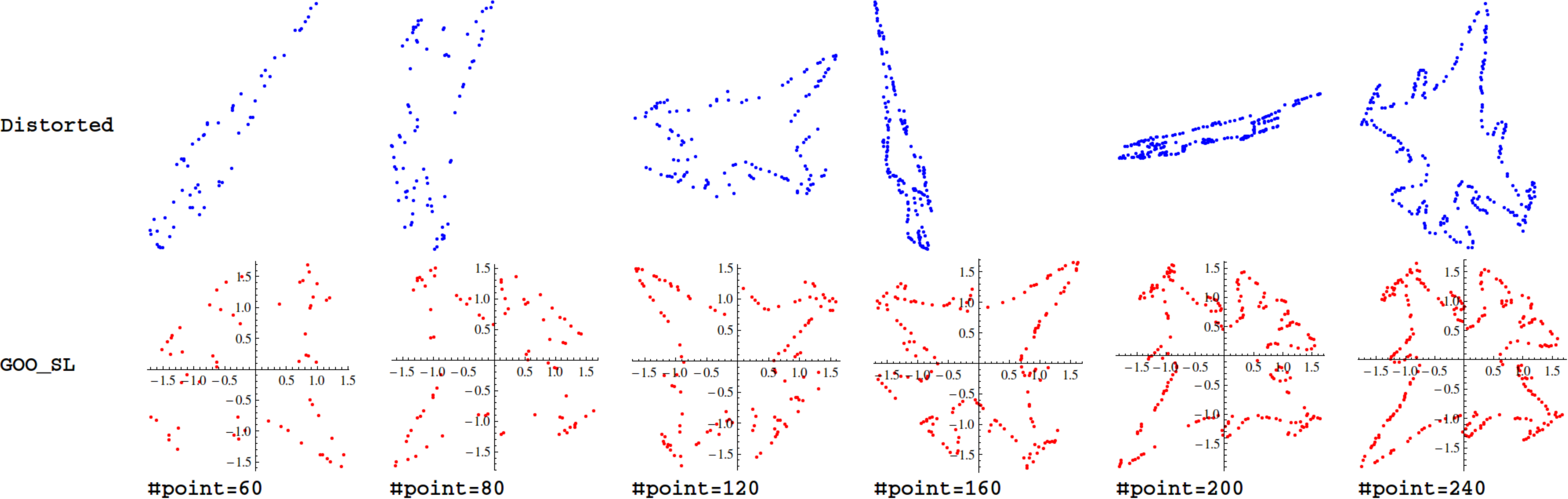}}
\subfigure{\includegraphics[height=48mm,
width=130mm]{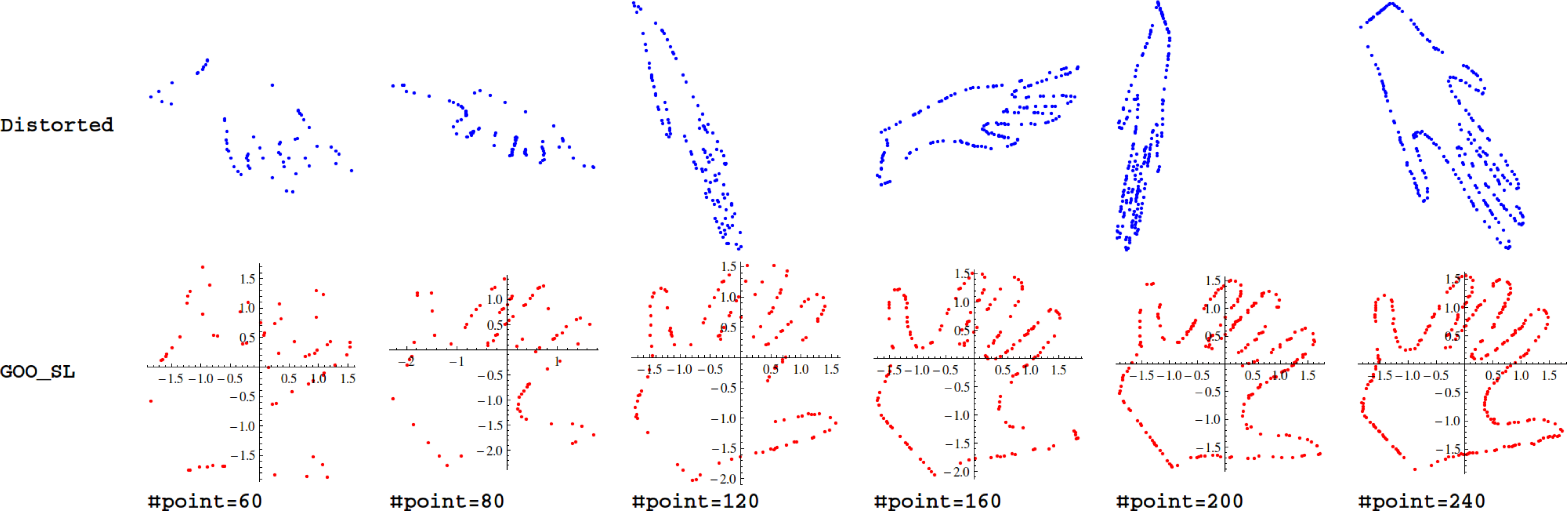}}

\end{center}
   \caption{The point clouds in rows marked with ``Distorted'' are the results
   of applying distortion generated by random special linear matrix to original
   point clouds.
   The point clouds in rows marked with ``GOO\_SL'' are after normalization by
   special linear group. From left to right, the sparser point clouds are
   generated by sampling from the rightmost densest point cloud respectively.
   It can be seen that for all three examples, though density varies, the
   shape of the normalized point clouds remains stable, modulo four possible
   orientations.}
\label{fig:sampling_rate}
\end{figure}

\section{Related work}
In this section we discuss the related work not yet covered in the previous sections.

An early example of GOO is a so-called quadratic assignment problem \cite{sahni1976p} where the
following optimization problem is studied:
\[
\inf_{\X\in\Pi_n} \tr(\W \X \D \X^\top) \text{,}
\]
where $\Pi_n$ is the permutation matrix group. Due to the combinatorial nature
of $\Pi_n$, QAP is NP-hard. In contrast,  we mainly work on
non-combinatorial matrix groups in this paper.

In \cite{zhang2012tilt}, a non-linear GOO  is used to find texture
invariant to rotation for 2D point cloud $\M\in\FB^{n\times 2}$:
\[
\inf_{\G\in\OG} \|\Rasterize(\Poly(\M \G))\|_* \text{.}
\]
As $\OG$ is a unit group, the optimization is well defined and the
induced matrix decomposition is found to be useful as a rotation-invariant
representation for texture. The same paper also considers finding
homography-invariant representation for texture for 2D point cloud
$\M\in\FB^{n\times 2}$:
\[
\inf_{\G\in\HG,\; \mu(\Poly(\lambda \M \G))=\const} \|\Rasterize[\Poly(\lambda
\M \G)]\|_*
\text{.}
\]
Note that here a coefficient $\lambda$ is intentionally
added to ensure $\mu$ measure of the point cloud be preserved \wrt the action
of $\G$.

In \cite{hu2013fast} the following formulation is used to get the Ky-Fan $k$-norm
\cite{horn1991topics} of a matrix $\M\in\FB^{m\times n}$ when $m\ge k$ and
$n\ge k$:
\[
\sup_{\G_1\in\FB^{m\times k}, \G_1^\top\G_1 = \I_k,\G_2\in\FB^{n\times k},
\G_2^\top\G_2 = \I_k} \tr(\G_1^\top \M \G_2)\text{.}
\]

Note that the above optimization is not a GOO when $k^2\neq m n$ as in that case
$\G_2^\top \otimes \G_1^\top \in \FB^{k^2\times m n}$ does not form a group.

\section{Conclusion}
\label{sec:conclusion}

In this paper, we have studied an optimization problem  over the group orbit
generated by action of group $\GM$ and referred to it as the \emph{Group Orbit
Optimization} (GOO).
We have shown that SVD/QR/LU/Cholesky decomposition  can be reformulated under
the GOO framework as in Theorem~\ref{theorem:decompostion-as-optimization}.
Moreover, we have used GOO to induce tensor decomposition in
Theorem~\ref{thm:unfolded-diagonal}. The unified framework of GOO for matrix
decomposition and tensor decomposition allows us to bridge them. In
particular, we have presented Lemma~\ref{lem:lifting}, which relates the infimum of
the tensor-based GOO with the infimum of GOO of the matrix unfolded to the
tensor. Finally, we have applied GOO to point cloud data to demonstrate the use of data
normalization in shape matching when objects are represented as point clouds.

Our work has demonstrated that the unified framework of GOO for
data normalization is both of theoretical interests in providing a new perspective on
matrix and tensor decompositions, and of practical interests in modeling and
elimination of distortions present in real world data.


\bibliography{sld_v3}
\bibliographystyle{siam}

\end{document}